\documentclass[11pt,twoside,a4paper]{amsart}
\usepackage{amssymb}
\usepackage{a4wide}
\usepackage{xcolor}
\date{\today}

\title[]
{Segre forms of singular metrics on vector bundles and Lelong numbers}

\def\la{\langle}
\def\ra{\rangle}

\def\e{{\epsilon}}
\def\w{\wedge}

\def\R{{\mathbb R}}
\def\C{{\mathbb C}}
\def\P{{\mathbb P}}

\def\B{{\mathcal B}}

\def\codim{{\rm codim\,}}

\def\Z{{\mathcal Z}}

\def\O{{\mathcal O}}

\def\L{{\mathcal L}}
\def\Q{{\mathbb Q}}

\def\L{{\mathcal L}}
\def\U{{\mathcal U}}
\def\Pk{{\mathbb P}}
\def\Ok{{\mathcal O}}

\def\1{{\bf 1}}
\def\La{{\mathcal L}}
\def\rank{{\rm rank \,}}
\def\GZ{{\mathcal {GZ}}}
\def\GGZ{{\mathcal {QZ}}}

\def\QZ{{\mathcal {QZ}}}
\def\QB{{\mathcal{QB}}}
\def\J{{\mathcal J}}

\def\k{{\kappa}}
\def\nbh{{neighborhood }}

\def\Zk{{\mathbb Z}}

\def\Ba{{\mathcal B}}
\def\QB{{\mathcal{QB}}}
\def\B{{\mathcal B}}
\def\Zc{{\mathcal Z}}

\def\qas{{qas}}

\def\s{{\bf s}}
\def\c{{\bf c}}
\def\I{{I}}

\DeclareMathOperator{\mult}{mult}

\DeclareMathOperator{\dv}{div}

\newtheorem{thm}{Theorem}[section]
\newtheorem{lma}[thm]{Lemma}
\newtheorem{cor}[thm]{Corollary}
\newtheorem{prop}[thm]{Proposition}

\theoremstyle{definition}

\newtheorem{df}[thm]{Definition}

\theoremstyle{remark}

\newtheorem{preremark}[thm]{Remark}
\newtheorem{preex}[thm]{Example}

\newenvironment{remark}{\begin{preremark}}{\qed\end{preremark}}
\newenvironment{ex}{\begin{preex}}{\qed\end{preex}}

\numberwithin{equation}{section}

\date{\today}

\author{Mats Andersson \& Richard Lärkäng}

\address{Department of Mathematics\\Chalmers University of Technology and the University of G\"oteborg\\
S-412 96 G\"OTEBORG\\SWEDEN}

\email{matsa@chalmers.se, larkang@chalmers.se}

\begin{document}

\begin{abstract}
Let $E\to X$ be a holomorphic vector bundle. We consider a class of a singular Hermitian metrics on $E$ with analytic singularities that contains all
Griffiths negative such metrics.
One can define, given a smooth reference metric $h_0$, a current $s(E,h,h_0)$ called the associated Segre form, which
defines the expected Bott-Chern class and coincides with the usual Segre form of $h$ where it is smooth.
We prove that $s(E,h,h_0)$ is the limit of the Segre forms of a sequence of smooth metrics if the metric is smooth outside the
degeneracy locus, and in general as a limit of Segre forms of metrics with empty degeneracy locus.
One can also define an associated Chern form $c(E,h,h_0)$. We prove that
the Lelong numbers of $s(E,h,h_0)$ and $c(E,h,h_0)$ are integers if the singularities are integral,
and non-negative for $s(E,h,h_0)$.
%We extend these results to a natural slightly larger
%class of singular metrics with analytic singularities.
\end{abstract}

\maketitle

\section{Introduction}
Inspired by the importance, during several decades,  of singular metrics on line bundles in
algebraic geometry, with origins in \cite{Dem1}, several efforts have been done to study
singular metrics on higher rank holomorphic vector bundles.
In particular, such metrics appeared in the work of Berndtsson-P\u aun, \cite{BePa},
in connection with pseudoeffectivity of twisted relative canonical bundles.
Contrary to the line bundle case, it is not clear what the curvature and connection matrix
for a singular metric should be, cf. \cite[Theorem~1.5]{Raufi}. This would appear to be an obstacle for defining associated Chern
forms (or other characteristic forms, like Segre forms), since a standard definition
is in terms of the curvature matrix.
However, for appropriate metrics, it has still been possible
to introduce such associated forms, or more precisely currents of order $0$,
see, e.g., \cite{Hajli}, \cite{Ina}, \cite{LRRS}, \cite{LRSW} and \cite{Xia}.
In this article, we further the study of such associated forms, and in slightly more general situations.

\smallskip

Assume that $E\to X$ is a holomorphic vector bundle over a complex manifold $X$
with a Griffiths negative singular metric $h$ with analytic singularities; this means that $\psi=\log|\cdot|^2_h$ is plurisubharmonic (psh) with analytic
singularities on $E\setminus 0_E$, see Section~\ref{qgnsec}. The latter condition means that locally
$%\begin{equation}\label{skatbo}
\psi=c\log|f|^2+ b,
$ %\end\equation}
where $f$ is a tuple of holomorphic functions, $c>0$ and $b$ is bounded.
We say that the singularities are integral (rational/real) if $c$ is a positive integer (rational/real number).
In \cite{LRSW}
was introduced an associated Segre form\footnote{We keep the notion 'Segre form' although in the singular case
$s(E,h,h_0)$ is in general a current (of order $0$).}
$s(E,h,h_0)$, slightly depending on a smooth  reference metric $h_0$ on $E$,
that represents the expected Bott-Chern class,
and coincides with the usual Segre form where the metric is smooth.

We will allow bundles with metrics such that $\psi=\log|\cdot|^2_h$ is quasi-plurisubharmonic (qpsh) with analytic
singularities on $E$.
If in addition $\psi$ has integral singularities, we say that
the metric, and the bundle equipped with such a metric, is \qas.
If the singularities are rational or real we say that the bundle/metric  is $\Q$-\qas\  or $\R$-\qas.
We mainly consider (integral)  \qas\ metrics,
which is practical in order to keep track on arithmetic information.
There are analogous results for $\Q$-qas and $\R$-qas metrics,  %
see Remark~\ref{mainQqasRcas}.

The definition of Segre form $s(E,h,h_0)$ in \cite{LRSW} extends to the case when $h$ is a qas metric on $E\to X$ and $h_0$ is a smooth reference metric. Let
$h_\e = h + \e h_0$, i.e.,
\begin{equation}\label{regular1}
|\cdot |^2_{h_\e}=|\cdot|^2_h+\e |\cdot |^2_{h_0}.
\end{equation}
If $h$ is smooth outside the degeneracy locus, then $h_\e$ are smooth and so the classical
associated Segre forms $s(E,h_\epsilon)$ are well-defined. In general $h_\e$ is qas with empty degeneracy locus,  and the same definition makes
sense using Bedford-Taylor products, see \eqref{spagat}.  Our first main result, Theorem~\ref{main1}, in particular states that
\begin{equation}\label{regular2}
s(E,h,h_0)=\lim_{\e\to 0} s(E,h_\e).
\end{equation}
%This provides further support that the definition of $s(E,h,h_0)$ is natural. % and also turns out to be useful.

\begin{ex}\label{modelex}
Assume that $g\colon E\to F$ is a holomorphic vector bundle morphism and $F$ has a smooth metric,
or a \qas\ metric with empty degeneracy locus. Then the metric $|\alpha|^2=|g(x)\alpha|^2$
on $E$ inherited from $F$ is \qas, see Corollary~\ref{sinister}.
\end{ex}

\begin{ex}
An example of a singular metric $e^{\psi}$ on a line bundle $L$ is given by
$\psi=\log(|\sigma_1|^2+\cdots +|\sigma_m|^2)$, where $\sigma_j$ are sections of $L^{-1}$.
It is covered by Example~\ref{modelex} via the mapping $L\to X\times \oplus_1^m\C $,
defined by $(\sigma_1,\ldots,\sigma_m)$. %KOLLA TECKEN!
Such metrics often occur in applications to complex algebraic geometry.
\end{ex}

We have that $s(E,h,h_0)=s_0(E,h,h_0)+s_1(E,h_0,h)+ \cdots$,  where
$s_k(E,h,h_0)$ are closed $(k,k)$-currents of order $0$.
They are not necessarily positive unless both $h$ and $h_0$ are
Griffiths negative. There are anyway well-defined multiplicities $\mult_x s_k(E,h,h_0)$
at each point $x$, defined in the same way as Lelong numbers, see Section~\ref{qzsection}.
Here is our first main theorem:

\begin{thm}\label{main1}
Assume that $E\to X$ is a holomorphic vector bundle equipped with a  \qas\ metric $h$,
and let $Z$ be the degeneracy locus of $h$.
Let $h_0$ be a smooth reference metric, and let $s(E,h,h_0)$ denote the
associated Segre form.

\smallskip
\noindent
(i) If $h_\epsilon$ is the metric defined by \eqref{regular1}, then \eqref{regular2} holds.

\smallskip
\noindent
(ii) There is a current $w$ on $X$ such that
\begin{equation}\label{ddcw}
dd^c w= s(E,h,h_0)- s(E,h_0).
\end{equation}

\smallskip
\noindent
(iii)
We have the decomposition
\begin{equation}\label{uppesittare}
s(E,h,h_0)=s'(E,h) + M^{E,h,h_0},
\end{equation}
where $M^{E,h,h_0}$ is closed and has support on $Z$ and $s'(E,h)=\1_{X\setminus Z} s(E,h,h_0)$
is independent of $h_0$. If $h$ is smooth on $X\setminus Z$, then
$s'(E,h)$ coincides with the usual Segre form of $h$ there.

\smallskip
\noindent
(iv) If $\pi\colon X'\to X$ is a modification,
then $\pi^*h$ is a \qas\ metric on $\pi^*E$, and
\begin{equation}\label{sformel}
\pi_* s(\pi^*E,\pi^*h,\pi^* h_0)=s(E,h,h_0).
\end{equation}

\smallskip
\noindent (v)
If $E'\to X$ is another holomorphic vector bundle, equipped with a smooth metric $h'$,
then
\begin{equation}\label{main31}
s(E'\oplus E,h'\oplus h,h'\oplus h_0)= s(E',h')\w s(E,h,h_0).
\end{equation}

\smallskip
\noindent
(vi)
For each integer $k\ge 0 $ there is a decomposition
\begin{equation}\label{siu0}
M_k^{E,h,h_0}=S^{E,h}_k+ N^{E,h,h_0}_k,
\end{equation}
where $S_k^{E,h}$ is the Lelong current of an analytic cycle of codimension $k$,
and the set of $x$ such that  $\mult_x s_k'(E,h)$ or $\mult_x N_k^{E,h,h_0}$
is nonvanishing, is contained
in an analytic set of codimension $\ge k+1$. The multiplicities of $s_k'(E,h)$, $S_k^{E,h}$ and $N_k^{E,h,h_0}$ are non-negative integers that are independent of $h_0$.

\smallskip
\noindent
(vii)
If $h'$ is a qas metric on $E \to X$ comparable to $h$, then $\mult_x M_k^{E,h,h_0}=\mult_x M_k^{E,h',h_0}$  and
$\mult_x s'_k(E,h)=\mult_x s_k'(E,h')$ for $x\in X$ and $k= 0,1,2, \ldots$.
In particular,  $S_k^{E,h}=S_k^{E,h'}$, $k=0,1,2,\ldots$.
\end{thm}

If the metric is $\Q$-\qas\  ($\R$-\qas),  then the same statements hold, except that the multiplicities in part (vi)
are rational (real).

\smallskip

\begin{remark}
\noindent (i) Equation \eqref{ddcw} means that $s(E,h,h_0)$ represents the Bott-Chern class of $E$.

\smallskip
\noindent (ii)
The properties (ii) to (iv) are (at least implicitly) proved in \cite{LRSW}.

\smallskip
\noindent(iii) Equality \eqref{main31} implies that
\begin{align*}
s'(E'\oplus E,h'\oplus h,h'\oplus h_0)&=s(E',h')\w s'(E,h,h_0)
%and
\text{ and } \\
M^{E'\oplus E,h'\oplus h,h'\oplus h_0} &=s(E',h')\w M^{E,h,h_0}.
\end{align*}

\smallskip
\noindent (iv) If $M_k^{E,h,h_0}$ is a positive current, \eqref{siu0} is precisely the Siu decomposition.
\end{remark}

\smallskip

The definition of a Griffiths negative singular metric was introduced in \cite{BePa}, and its dual metric is  by definition Griffiths positive. Although most attention in the literature
is paid to positive metrics, we focus in this article on qas metrics which are slightly more general than negative metrics with integral analytic singularities.
This leads to simpler formulations, in particular less signs. By considering duals of such metrics, one obtains corresponding statements concerning
Griffiths positive metrics, cf. Remark~\ref{remLRSWcomparison}.

\smallskip

Recall that in algebraic geometry, characteristic classes of vector bundles such as the Chern class,
are defined as endomorphisms on Chow groups, i.e., the $\Zk$-module
of analytic cycles $\Zc(X)$ modulo rational equivalence.

Inspired by \cite{AESWY} we introduce a $\Zk$-module $\QZ(X)$ of currents of order $0$
that contains the module $\Zc(X)$ of analytic cycles as a submodule, if
cycles are identified with their associated Lelong currents. In particular the $(0,0)$-current $1$ is an element in $\QZ(X)$.
Any element $\QZ(X)$ of a fixed dimension has a well-defined multiplicity (Lelong number) that is an integer, at each point.
We also introduce certain quotient classes $\QB(X)$ that, contrary to the Chow classes,
keep local information such as multiplicities; this fact is crucial in this paper.

%These numbers that only depends on its class in $\QB(X)$. This is in contrast to Chow classes that

\smallskip
A second main result is that the Segre form $s(E,h,h_0)$ extends to an endomorphism
$\s(E,h,h_0)$ on $\QZ(X)$ such that $\s_0(E,h,h_0)=I$, the identity, and $\s(E,h,h_0) 1=s(E,h,h_0)$.
An important fact is that $\s(E,h,h_0)$ induces a well-defined endomorphism
$\s(E,h)$ of $\QB(X)$ that is independent of $h_0$.
There is an analogue of Theorem~\ref{main1} for $\s(E,h,h_0)$, see Theorem~\ref{main11}.
Since $\s(E,h,h_0)-I$ has positive degree, %and therefore
\begin{equation}\label{tyska}
\c(E,h,h_0)=\sum_{\ell=0}^\infty (-1)^\ell (\s(E,h,h_0)-\I)^\ell
\end{equation}
is a well-defined endomorphism such that
\begin{equation}\label{stork}
\c(E,h,h_0) \s(E,h,h_0)=\s(E,h,h_0) \c(E,h,h_0)=\I.
\end{equation}
It is natural, cf.~\eqref{cformel},  to say that $c(E,h,h_0):=\c(E,h,h_0)1$ is the Chern form associated with $(E,h,h_0)$.
The same definition of $c(E,h,h_0)$ appeared in \cite{LRSW} but expressed in a different way.

\begin{thm}\label{main2}
Let $(E,h,h_0)$ be as in Theorem~\ref{main1}.
Properties (ii)-(vii) of Theorem~\ref{main1} hold if $s(E,h,h_0)$, $s'(E,h)$, and $M^{E,h,h_0}$ are replaced by $c(E,h,h_0)$, $c'(E,h)$, and $\1_Z s(E,h,h_0)$,
except that the multiplicities are not non-negative.
\end{thm}

The analogue of \eqref{regular2} is not true for $c(E,h,h_0)$, see Section~\ref{remex}.
%In a forthcoming paper it will be showed that the analogue of \eqref{regular2} is not true in general for $c(E,h,h_0)$.

As already mentioned, the starting point for this paper is \cite{LRSW} and the ideas there.
The regularization \eqref{regular1} relies on results in \cite{ABW, Bl}.
The arithmetic analysis of the Lelong numbers
relies on ideas in \cite{Asommar}, and the $\Zk$-modules $\QZ(X)$ and $\QB(X)$ that
are modelled after the somewhat smaller modules $\GZ(X)$ and $\B(X)$ in \cite{AESWY}.

\subsection*{Acknowledgements}

We are grateful to H\aa kan Samuelsson Kalm, Antonio Trusiani and David Witt Nystr\"om for valuable discussions.

\section{Preliminaries}

Throughout this article $X$ is assumed to be a complex manifold.
Given a holomorphic vector bundle $E\to X$ we let $p\colon \Pk(E)\to X$ denote the projectivization,
so that the fiber above $x$ is the set of all lines\footnote{There is another common convention in the literature, where
the fiber instead is all hyperplanes through the origin.} through the origin in $E_x$.
Then the pullback $p^*E\to \Pk(E)$ has the tautological sub(line)bundle
$L_E\to \Pk(E)$ whose fiber over $(x,[\alpha])\in \Pk(E)$ is the line in $E_x$ determined by
$\alpha$.

If $E$ has a smooth Hermitian metric $h$, then $p^* E$ and hence $L_E$ inherit
smooth Hermitian metrics.
Notice that a local section $s$ of $L_E$ has the form $s=\tilde s(x,\alpha) \alpha$, where
$\tilde s$ is a holomorphic function on $E\setminus 0_E$ which is $-1$-homogeneous in $\alpha$. Thus any local section can be identified with such an $\tilde s$.
The induced norm is
\begin{equation}\label{snorm}
|s|^2_{e^\psi}:=|\tilde s|^2|\alpha|^2_h.
\end{equation}
In a local frame for $L_E$, say $1/\alpha_j$, the section $s$ is represented by the function $\tilde s \alpha_j$, and
from \eqref{snorm} we have that
$$
|s|^2_{e^\psi}=|\tilde s\alpha_j|^2e^{\log(|\alpha|_h^2/|\alpha_j|^2)}.
$$
Thus the metric $e^\psi$ on $L_E$ with respect to the local frame $1/\alpha_j$ is given by
\begin{equation}\label{phiform}
\psi=\log(|\alpha|^2_h/|\alpha_j|^2).
\end{equation}
Therefore $dd^c \log |\alpha|_h^2$ is a well-defined form on $\Pk(E)$ and
$c_1(L_E,e^\psi)=-dd^c\log|\alpha|^2_h$ is the first Chern form.
Moreover,  %Now %so that the
$$
s(L_E,e^\psi)=\frac{1}{c(L_E,e^\psi)}=\frac{1}{1+c_1(L_E,e^\psi)}
%=\frac{1}{1-s_1(L_E,e^\psi)}
=\sum_{\ell=0}^\infty (dd^c\log|\alpha|_h^2)^\ell.
$$
The Segre form of $(E,h)$ is defined as
\begin{equation}\label{segdef}
s(E,h)=p_* s(L_E,e^\psi)=p_*\sum_{\ell=0}^\infty (dd^c\log|\alpha|_h^2)^\ell,
\end{equation}
and it is a smooth form since $p$ is a proper submersion.
We write $s(E,h) = \sum s_k(E,h)$, where $s_k(E,h)$ has bidegree $(k,k)$.
It turns out that $s_0(E,h)=1$ and thus $s(E,h)-1$ has positive bidegree. We get the Chern form
\begin{equation}\label{cformel}
c(E,h)=\frac{1}{s(E,h)}=\sum_{\ell=0}^\infty (-1)^\ell (s(E,h)-1)^\ell.
\end{equation}
By \cite{Mour}, see also \cite{Div}, this  definition of $c(E,h)$ coincides
with the definition as the sum of the elementary symmetric
polynomials acting on ($i/2\pi$ times) the curvature tensor $\Theta(h)=D_h^2$ of $(E,h)$, where
$D_h$ denotes the Chern connection associated with $h$.  The Segre form is then defined as
$s(E,h)=1/c(E,h)$.

\subsection{Qpsh functions with analytic singularities}

%\begin{df}
Let $Y$ be a complex manifold.
Recall that a real-valued function $u$ is qpsh (quasi-plurisubharmonic) if locally
$u=v+a$, where $v$ is psh and $a$ is smooth.
As already mentioned, $u$ has (integral) analytic singularities if locally
\begin{equation}\label{totem1}
u=c\log|f|^2+ b,
\end{equation}
where $b$ is bounded and $c$ is a positive (integral) number. %$>0$.

\begin{lma}\label{lemur}
Assume that $W$ is a connected complex manifold. If $\tau\colon W\to Y$ is holomorphic and $u$ is qpsh with (integral)
analytic singularities on $Y$, then $\tau^*u$ is qpsh with (integral) analytic singularities in $W$.
Possibly $\tau^*u$ is identically $-\infty$.
\end{lma}

\begin{lma}\label{proviant}
If $\psi_1$ and $\psi_2$ are qpsh (with integral analytic singularities),
then $\psi_1+\psi_2$ and
$
\log(e^{\psi_1}+e^{\psi_2})
$
are qpsh (with integral analytic singularities).
\end{lma}

Lemma~\ref{lemur} follows directly from the definitions. Lemma~\ref{proviant}
is well-known for psh $\psi_j$ (with integral analytic singularities). It is readily
checked that it also holds for qpsh (with integral analytic singularities), using
the characterization that functions $\psi_j,\ j=1,2,$ are qpsh (with integral
analytic singularities) if and only if locally there is a smooth function $\phi$ such that
$\psi_j':=\psi_j + \phi,\ j=1,2,$ are psh (with integral analytic singularities).

%\subsection{Normal currents}\label{normal}
%Recall that a $(k,k)$-current $\mu$ is normal if both $\mu$ and $d\mu$ have order $0$. If such a $\mu$  has support on
%a subvariety $V$ of codimension larger than $k$, then $\mu=0$, see \cite[Ch.~III, Corollary 2.11]{Dem}. Notice that if
%$i\colon V\to Y$ is a subvariety and $\mu$ is normal on $V$, then $i_*\mu$ is normal.

%\subsection{A local regularization of a bounded qpsh functions}

%\smallskip
%Later on we will need the following somewhat technical
%modification of the standard local regularization
%of bounded qpsh functions through convolution.
%that which will be suitable for approximating \qas\ metrics
%with empty unbounded locus.

%\begin{prop}
%If $\psi$ is a bounded qpsh function then locally there are smooth qpsh functions $\psi_\e$ decreasing to $\psi$
%such that
%\begin{equation}
 %   (dd^c \psi_\e)^k \to (dd^c \psi)^k.
%\end{equation}
%\end{prop}

\section{\qas\ bundles}\label{qgnsec}

We say that a function is \qas\  if it is qpsh and has integral analytic singularities. We can assume then that $c=1$ in \eqref{totem1}.

\begin{df}
%We say that a real function is \qas\ if it is qpsh with integral analytic singularities.
A singular metric $h$ on a holomorphic vector bundle $E\to X$ is \qas\ if the function $\alpha\mapsto \log|\alpha|^2_h$ is
\qas\ on the total space $E$.
In that case $(E, h)$ is said to be \qas.
The degeneracy locus of $h$ is the set $Z\subset X$ of points $x$ such that $h(x)$ is degenerate.
\end{df}

Notice that a qas metric $h$ is non-degenerate if and only if
$\alpha\mapsto\log|\alpha|^2_h$ is locally bounded on $E\setminus 0_E$.

\begin{ex}\label{plaskdamm}
Let $L$ be a line bundle with metric $h=e^\psi$. We claim that $(L,e^\psi)$ is qas if and only if $\psi$ is \qas\ on $X$.
Given a local frame,  $L=U\times \C_\alpha$ and
$\log |\alpha|^2_{h(x)}=\log|\alpha|^2+\psi(x)$. If this function is \qas\ on $L$, then by Lemma~\ref{lemur}
its pullback to $U=U\times \{1\}$, i.e., $\psi$, is \qas.
% if and only if $\psi(x)$ is
%\qas.  In fact, the 'only if'-part follows by pulling back to $U=U\times \{1\}$, cf.~Lemma~\ref{lemur}.
Conversely, if $\psi(x)$ is \qas\, then $\log|\alpha|^2+\psi(x)$ is \qas\ on $U\times \C_\alpha$,
cf.~Lemmas~\ref{lemur} and \ref{proviant}.
%The converse is simple (it follows from the two lemmas).
\end{ex}

Sometimes we suppress $h$ and say that $E$ is a \qas\ bundle, and write $|\alpha|$ rather than $|\alpha|_h$.

\begin{lma}\label{knepig0}
(i)  If $\tau\colon W\to X$ is holomorphic and $E\to X$ is \qas, then $\tau^*E\to W$ is \qas.

\smallskip
\noindent
(ii) If $g\colon E\to F$ is a bundle morphism and $F$ is \qas, then $E$ is \qas\ if equipped with the induced
metric $|\alpha|:=|g(x)\alpha|$.

\smallskip
\noindent
(iii) If $(E_j,h_j)$ are qas bundles over $X$,  $j=1,2$, then $(E_1\oplus E_2, h_1\oplus h_2)$ is \qas.

\smallskip
\noindent
(iv)
If $E\to X$ has two \qas\ metrics $h_1$ and $h_2$, then $h_1+h_2$ is a \qas\ metric on $E$.
%That is,
%$|\alpha|^2=|\alpha|^2_{h_1}+|\alpha|^2_{h_2}$ is \qas.

\smallskip
\noindent
(v)
If $L$ is a \qas\ line bundle, and $E$ is a \qas\ vector bundle, then $E\otimes L$ is \qas.
\end{lma}

\begin{proof}
If $\tau\colon W\to X$, then $|\tau^*\alpha|^2_{\tau^*h}=\tau^*|\alpha|^2_h$ so (i) follows from
Lemma~\ref{lemur}.
Part (ii) also follows from Lemma~\ref{lemur} by considering the holomorphic mapping $E\to F$, $(x,\alpha)\mapsto (x,g(x)\alpha)$.

Locally $E_j=X\times \C_{\alpha_j}^{r_j}$. If $\log |\alpha_j|^2$ is \qas\ on $E_j$, $j=1,2$,
then so is its  pullback to $E_1\oplus E_2$.
%$\log|\alpha_j|^2$ is \qas\ on $E_1\oplus E_2$.
By Lemma~\ref{proviant},
$\log(|\alpha_1|^2 +|\alpha_2|^2)$ is \qas\ on $E_1\oplus E_2$. Thus (iii) holds.
We get (iv) by applying (ii) and (iii) to $(E,h_1)$, $(E,h_2)$ and
the mapping
$E\to E\oplus E, \   \alpha\mapsto (\alpha,\alpha).$
For (v), note that if $e_L$ is a local frame for $L$ and $e^{\psi_L}$ the corresponding representation
of the metric, then
$\log |\alpha \otimes e_L|^2 = \log |\alpha|^2 + \psi_L$, so
(v) follows from Lemma~\ref{proviant}.
\end{proof}

\begin{remark}
We say that a function is $\Q$-\qas/$\R$-\qas\ if
it is qpsh with analytic singularities, and the local positive constants $c$ in \eqref{totem1}
are rational/real.  We say that a vector bundle is
$\Q$-\qas/$\R$-\qas\ if $\alpha \mapsto \log |\alpha|^2$ is $\Q$-\qas/$\R$-\qas. %of the corresponding class.
Lemma~\ref{lemur} holds for $\Q$-\qas\ and $\R$-\qas\ functions,
and consequently Lemma~\ref{knepig0} (i)-(ii) holds for $\Q$-\qas\ and $\R$-\qas\ bundles.
Lemma~\ref{proviant} also holds for $\Q$-\qas\ functions, but not
for $\R$-\qas\ functions in general. In fact,  % since $c_1 \log |f_1|^2 + c_2 \log |f_2|^2$ and
$\log (|f_1|^{2c_1}+|f_2|^{2c_2})$ does not have analytic singularities if $c_1,c_2$ are arbitrary real numbers. % not necessarily have analytic singularities
%if $c_1$ and $c_2$ are linearly independent over $\Q$.
%
Therefore, the analogues of Lemma~\ref{knepig0} (iii)-(v) hold for $\Q$-\qas\ bundles,
but not necessarily for $\R$-\qas\ bundles.
\end{remark}

\begin{lma}\label{orkan}
Let $E$ be a holomorphic vector bundle on $X$ with a smooth Hermitian metric $h$.
Any point $x_0 \in X$ has a \nbh $U$ such that
$h' = h e^{A|x|^2}$ is Griffiths negative if $A$ is large enough, where
$x$ denotes a local coordinate system near $x_0$.
\end{lma}

\begin{proof}
If the curvature $\Theta_h$ of $h$ is $\sum \Theta_{jk} dx_j{\wedge}d\bar{x}_k$,
then $h$ is Griffiths negative if $\sum (\Theta_{jk} s, s) v_j \bar{v}_k \leq 0$ for any section $s$ of $E$ and
tangent vector $v = \sum v_j \frac{\partial}{\partial x_j}$.
Notice that $\Theta_{h'} = \Theta_{h} - I_EA\omega_x $,
where $\omega_x$ is the standard Kähler form on $U \subset \C^N_x$ and $I_E$ is the identity endomorphism on $E$.
By possibly shrinking $U$, we may choose $A$ such that
$|\sum (\Theta_{jk} s, s) v_j \bar{v}_k| \leq A|s|^2 |v|^2$.
Then $h'$ is Griffiths negative.
\end{proof}

The statement means that locally there is a negative trivial line bundle $L_A\to U$ such that $E\otimes L_A$ is Griffiths negative.
It is well-known that if $(E,h)$ is Griffiths negative, then $\log|\alpha|^2_h$ is psh on $E$, see, e.g., \cite{BePa}. Thus we have

\begin{cor}\label{sinister}
Any vector bundle with a smooth Hermitian metric is \qas.
\end{cor}

\begin{lma}\label{knepig1}
%Assume that
Assume that $h$ is a qas metric. For any $x_0\in X$ there is a \nbh $U\subset X$, holomorphic $r$-tuples $f_k(x)$, $k=1,\ldots, m$, and $C>0$, such that in $E|_U\simeq U\times \C^r_\alpha$, %$k=1,\ldots, m$
\begin{equation}\label{garn}
(1/C)|\alpha|_{h(x)} \le \sum_{k=1}^m |f_k(x)\cdot \alpha)|\le C |\alpha|_{h(x)}.
\end{equation}
%in $U\times \C^r$.
\end{lma}

\begin{proof}
Let $x_0 = 0 \in \C^n$ and let $|\alpha|$  be the standard norm on $\C^r_\alpha$. The assumption that $\log|\alpha|^2_{h(x)}$ has analytic singularities means that
 %there are % in a local embedding, by the assumption of analytic singularities,
there are holomorphic functions $F_k(x,\alpha)$ in a set $|x|\le c, |\alpha|\le c,\ c > 0$, such that
\eqref{garn} holds there with $F_k(x,\alpha)$ instead of $f_k(x)\cdot\alpha$. Since $F_k(x,0)=0$ we have that $F_k(x,\alpha)=f_k(x)\cdot \alpha+F'_k$, where
$F'_k$ is $\Ok(|\alpha|^2)$. % (where $|\alpha|_0^2$ denotes the trivial metric on $\C^r_\alpha$).
We write the tuple as $F=f(x)\cdot\alpha+F'$. For $x\le c$, $|\alpha|=c$ and $|\lambda|\leq 1$, %we have
$
F(x,\lambda \alpha)= \lambda f(x)\cdot\alpha+\Ok(|\lambda|^2).
$
Thus
$$
|\lambda||\alpha|_{h(x)}\le C|F(x,\lambda \alpha)|\le C| \lambda|| f(x)\cdot\alpha|+C|\lambda|^2.
$$
%so that
%$$
%|\alpha|_{h(x)}\le C|F(x,\lambda \alpha)|\le C| F'(x,\alpha)|+C|\lambda| |F''(x,\alpha)|,
%$$
Letting $\lambda\to 0$ we get
$$
|\alpha|_{h(x)}\le C| f(x)\cdot \alpha|, \quad |\alpha|=c, \ |x|\le c.
$$
By homogeneity this holds for all $\alpha$. In the same way we get the opposite inequality. %estimate.
Thus \eqref{garn} holds with $f(x)\cdot\alpha$ instead of $F(x,\alpha)$.
Moreover,
$f(x)\cdot\alpha$ is defined in the whole set $U\times \C^r$ and \eqref{garn} holds there by homogeneity.
%%%
 \end{proof}
In $U$ thus $|\cdot|_h^2 $ is comparable to the metric  in Example~\ref{modelex}, if
$F=U\times \C^m$ with trivial metric, and $g=(f_1,\ldots, f_m)$.

\begin{prop}\label{taubunt}
If $E\to X$ is \qas, then $L_E\to \Pk(E)$ is \qas.
\end{prop}

\begin{proof}
By Lemma~\ref{knepig0}~(i), $p^*E\to \Pk(E)$ is \qas, and by (ii) its subbundle $L_E$ is \qas.
\end{proof}

%Notice by the way that $E$ is Griffiths negative, i.e., $\psi$ is psh, if and only if
%$e^\psi$ is a negatively curved metric.

\begin{remark}\label{viktig}
Almost all results in this paper assuming that $E$ is \qas\ only depend on that $L_E$ is \qas.
At a few places, Lemma~\ref{knepig0}~(ii)-(iii) and Lemma~\ref{knepig1},
and the non-negativity of multiplicities in Theorem~\ref{main1}, %t the multiplicities are non-negative,
we use the (possibly) stronger assumption that $E$ is \qas.
%The remainder of
%the results in the article which are formulated for \qas\ vector bundles $E$ indeed
%hold if $L_E$ is assumed to be \qas.
%
Also if $L_E$ has a psh metric with analytic singularities, then
the multiplicities are non-negative. Indeed, locally we can choose a negative reference metric;
then $s_k(E,h,h_0)$ is a positive $(k,k)$-current, and hence it has
non-negative Lelong numbers.
%
%
%By the arguments in the proof of
In \cite{LRSW}, the standing assumption is that $E$ is equipped with a metric $h$ such that
$L_E$ is psh with analytic singularities.
By the arguments in the proof of  \cite[Proposition~5.4]{LRSW} it follows that
 % this is equivalent to $\log |\alpha|_h^2$
%being psh with analytic singularities on $E \setminus 0_E$.
%The same arguments show that
$L_E$ is qas if and only if $\log |\alpha|_h^2$ is qas on
$E \setminus 0_E$.
%if and only if
%
%the corresponding statement when psh is replaced by qpsh, i.e.,
%for \qas\ metrics.
Following the proof of \cite[Proposition~5.2]{LRSW}, one can
show that if $L_E$ is \qas, then $\log |\alpha|_h^2$ is
qpsh on all of $E$. However we do not know if $\log |\alpha|_h^2$
has analytic singularities over $0_E$, so $E$ being \qas\ is possibly a stronger
assumption, unless $\rank E=1$, cf.~Example~\ref{plaskdamm}.
% that $E$ is \qas.
%
%The results mentioned above where we need to assume that $E$ is \qas\ do not appear
%in \cite{LRSW}. However,
%However, if $L_E$ has a psh metric with analytic singularities, then
%the Lelong numbers of Segre forms are non-negative. Indeed, since the Lelong numbers
%do not depend on the choice of reference metric, we may may locally choose a negatively curved
%reference metric, and then $s_k(E,h,h_0)$ is a positive $(k,k)$-current, and hence it has
%non-negative Lelong numbers.
%
%If the following generalization of Lemma~\ref{orkan} to the case when $L_E$ is \qas\ holds,
%then it would imply that $E$ may be equipped with a singular metric $h'$ comparable to $h$
%which is such that induced metric on $L_E$ is psh with analytic singularities, and
%in particular, $s_k(E,h,h_0)$ would have non-negative Lelong numbers also under this assumption.
%The generalization of Lemma~\ref{orkan} is as follows: If $E$ is equipped with a singular
%metric $h$ such that $L_E$ is \qas, then $h e^{A|x|^2}$ is a Griffiths negative singular metric
%in a neighborhood of $x_0$ for $A$ large enough, i.e., the induced metric on $L_E$ is psh.
%We do not know whether it holds.
\end{remark}

\section{Generalized cycles and quasi-cycles}\label{qzsection}

Let $Y$ be a complex manifold or a reduced analytic space.
In \cite{AESWY} was introduced the $\Zk$-module $\GZ(Y)$ of generalized cycles,
which are closed currents on $Y$. This module is an extension of the module $\Z(Y)$ of analytic cycles,
i.e., there is an injective mapping $i\colon \Z(Y)\to \GZ(Y)$, if the elements in
$\Z(Y)$ are identified with their associated Lelong currents.
It was shown that many properties of $\Z(Y)$ such as unique decomposition into irreducible components,
well-defined integer multiplicities at each point etc., hold for $\GZ(Y)$ as well.

Recall that if $Y$ is a reduced analytic space, then smooth forms on $Y$ are
defined in terms of restrictions of smooth forms in local embeddings, see, e.g.,
\cite{ASSmooth}, and currents are  dual elements as in the smooth case. %in an appropriate way. %dual of smooth compactly supported forms.

Let $\QZ(Y)$ be the $\mathbb{Z}$-module of locally finite sums of currents of the form
$\tau_* \gamma$, where $\tau \colon W \to Y$ is a proper holomorphic mapping, $W$ is a complex manifold,
and $\gamma$ is a Bedford-Taylor product, \cite{BT},
\begin{equation} \label{na}
	\gamma=dd^c b_1 \wedge \cdots \wedge dd^c b_p,
\end{equation}
on $W$, where $e^{b_j}$ are metrics on holomorphic line bundles $L_j$ over $W$
such that $b_j$ are locally bounded and qpsh.
We say that the elements in $\QZ(Y)$ are {\it quasi-cycles} on $Y$.

\begin{remark}
If we assume all $b_j$ be smooth we get precisely the module $\GZ(Y)$,
cf.~\cite[Remark~3.1]{AESWY}, which is thus a submodule of $\QZ(Y)$.
\end{remark}

\begin{remark}
One can just as well consider the corresponding $\Q$-module or $\R$-module that we denote by
$\QZ_\Q(Y)$ and $\QZ_\R(Y)$, respectively.
%It is natural to refer to the elements as rational
%and real quasi-cycles, respectively.
\end{remark}

The product \eqref{na} is a closed current of
order $0$ so it follows that any $\mu\in\QZ(Y)$ is a closed current of order $0$.
%It is well-known that if $V$ is an analytidIn particular, such currents may be multiplied with characteristic functions $\1_{V}$,
%if $V$ is an analytic subset of $W$.
If $\tau \colon Y' \to Y$ is a proper map and $\mu \in \GGZ(Y')$, then $\tau_* \mu \in \GGZ(Y)$.

%\smallskip
Most results about  $\GZ(Y)$ in \cite{AESWY} extend to $\QZ(Y)$,
with the same arguments, because of the following
three properties of mixed Monge-Ampère products of the form \eqref{na}:

\begin{lma}\label{elbas}
	Let $e^{b_1},\ldots,e^{b_k}$ be metrics on line bundles $L_1,\dots,L_k$ on a
    complex manifold $W$ such that $b_1,\dots,b_k$ are locally bounded and qpsh.

\smallskip
\noindent (i)	If $i\colon V\to W$ is an analytic subspace of positive codimension, then
	\begin{equation} \label{eq:BTprop1}
		\1_V dd^c b_1 \wedge \cdots \wedge dd^c b_k = 0.
	\end{equation}

\smallskip
\noindent
(ii)	If $i\colon V \to W$ is as in (i), then
	\begin{equation} \label{eq:BTprop2}
		dd^c b_1 \wedge \cdots \wedge dd^c b_k \wedge [V] =
		i_*( dd^c i^*b_1 \wedge \cdots \wedge dd^c i^*b_k).
	\end{equation}

\smallskip
\noindent
(iii)	If $W$ is compact and $k=\dim W$, then
	\begin{equation} \label{eq:BTprop3}
		\int_W dd^c b_1 \wedge \cdots \wedge dd^c b_k = \int_W c_1(L_1) \wedge \cdots \wedge c_1(L_k).
	\end{equation}
\end{lma}

Here $c_1(L_j)$ denotes the Chern class of $L_j$, which can be represented by the
first Chern form associated with any smooth Hermitian metric on $L_j$.

\begin{proof}
It is well-known that products like \eqref{na},
do not charge pluripolar sets, see,  e.g., \cite[Proposition III.3.11]{Dem}.
In particular \eqref{eq:BTprop1} holds.
The equality \eqref{eq:BTprop2} is a local statement and follows by a suitable regularization,
see, e.g., \cite[Lemma~3.4]{AWpascal}.
The last equality, \eqref{eq:BTprop3}, is well-known in case all $b_j$ are smooth. Here is
a sketch of an argument in our case.
If $\beta_j$ are smooth metrics, then $b_j-\beta_j$ are global functions.
We prove by induction that \eqref{eq:BTprop3} holds if $b_1,\ldots, b_\ell$ are replaced by $\beta_1,\ldots \beta_{\ell}$ for $1\le \ell\le k$.
Assume it is true for	 $\ell-1$. Then by the commutativity of Monge-Ampère products it is enough to verify, with
$u=b_\ell-\beta_\ell$, that
$$
dd^c (uT)=dd^c b_\ell\w T- dd^c\beta_\ell\w T
$$
if $T$ is any positive closed current. This is a local statement and follows from \cite[Proposition~III.4.9]{Dem} since locally $u=\hat b_\ell-\hat\beta_\ell$,
where $\hat b_\ell$ is psh and bounded, and $\hat\beta_\ell$ is smooth.
Now the induction step follows from Stokes' theorem.
\end{proof}

Using \eqref{eq:BTprop1}, it follows, cf.~\cite[Lemma~3.2]{AESWY},  that
if $V \hookrightarrow Y$ is an analytic subspace and $\mu \in \GGZ(Y)$,
then $\1_V \mu \in \GGZ(Y)$.
Indeed, if
\begin{equation}\label{muform}
\mu = \sum (\tau_j)_* \gamma_j,
\end{equation}
where $\tau_j \colon W_j \to Y$, then
\begin{equation} \label{eq:restrGGZ}
    \1_V \mu = \sum_{\tau_j(W_j) \subset V} (\tau_j)_* \gamma_j.
\end{equation}

\begin{prop}\label{tax}
Assume that $i\colon V\to Y$ is an analytic subspace. Then the image of $i_*\colon \QZ(V)\to \QZ(Y)$
is precisely the elements in $\QZ(Y)$ that have Zariski support on $V$.
\end{prop}

The proof is the same as for the analogous statement for $\GZ(Y)$,
\cite[Proposition~3.6]{AESWY}.

\begin{ex}\label{pudel}
If $\mu\in \QZ(Y)$ has support at $x\in Y$, then since it has order $0$ it must be a number times $[x]$, the point evaluation current at $x$.
In view of Proposition~\ref{tax}
this number is a finite sum of numbers $\alpha=\int \tau_*\gamma$, where $\tau(W)=\{x\}$. Now
$$
\alpha=\int_Y \tau_*\gamma=\int_W \gamma,
$$
which is an intersection number, cf.~\eqref{eq:BTprop3}, and hence an integer.
\end{ex}

\subsection{Monge-Ampère type products as endomorphisms on $\QZ(Y)$} \label{MAhom}
Assume that $u$ is a \qas\ function on $Y$. %Notice that two different local representations
 %\eqref{totem1}  (with $c=1$) define the same integral closed ideal sheaf.
 %Thus $u$ defines a global coherent integral closed ideal sheaf $\J\to Y$.
Let  $f$ be the tuple in a local representation  \eqref{totem1}  in $U\subset Y$ (with $c=1$) and let
$\pi\colon U' \to U$ be the normalization of the blowup of $Y$ along (the ideal generated by) $f$.
If $\sigma$ is a section of the line bundle $\L\to U'$ defining the
exceptional divisor in $U'$, then $\pi^* f=\sigma f'$ where $f'$ is a non-vanishing tuple of sections of $\L^{-1}$.
If $\tilde f$ is the tuple from another representation
\eqref{totem1} in $U$, then $|\tilde f|\sim |f|$ and therefore this modification is also the normalization of the blowup along $\tilde f$.
This follows from the minimality property of such a modification.
%If $\sigma$ is a section of $\L\to U'$ defining the
%exceptional divisor in $U'$, then $\pi^* f=\sigma f'$ where $f'$ is a non-vanishing tuple of sections of $\L^{-1}$.
%In particular one has a well-defined exceptional divisor
%and if $\sigma$ defines this
Therefore we get a global normal modification $p\colon Y'\to Y$  such that any $f$ occurring in some local \eqref{totem1} is
principal. (If wanted, by taking a further modification, one can assume that $Y'$ is smooth but this is not necessary.)
 %
 %At a given point $x$ then $\J_x$ is the sheaf of germs of holomorphic
 %functions $\phi$  at $x$ such that $|\phi|\le C|f|$ for some constant $C$ in a \nbh of$x$, and $f$ is a tuple in \eqref{totem1}.
%Let $\pi\colon Y'\to Y$ be a modification which principalizes this ideal sheaf.
We thus get a global section $\sigma$ of a
holomorphic line bundle $\La\to Y'$
such that, if we equip $\La$ with some Hermitian norm,
%
%the ideals associated to $\pi^* u$, i.e.,
%such that in any local representation \eqref{totem1}, $f$ defines a principal ideal.
%Such a modification indeed exists. In the case of integral analytic singularities,
%note first that we may always choose $c=1$ in \eqref{totem1}. If we then have two different
%such representations with $f$ and $f'$, then $\log(|f'|^2/|f|^2)$ is locally bounded,
%i.e., $|f'|\sim |f|$. If we let $\J$ denote the integral closure of the ideal generated by $f$,
%we thus obtain a well-defined coherent ideal sheaf on $Y$. By blowing up $\J$
%(followed by a resolution of singularities if one wants $Y'$ to be smooth), one obtains
%such a modification.
 %%
%%
\begin{equation}\label{lemur2}
\pi^*u=\log|\sigma|^2+a,
\end{equation}
where $a$ is qpsh and locally bounded.
By the
Poincar\'e-Lelong formula, $dd^c\log|\sigma|^2= [\dv \sigma]-c_1(\La)$ and hence
\begin{equation}\label{kontor}
dd^c\pi^*u=[\dv \sigma]+\nu,
\end{equation}
where $\nu=-c_1(\La)+a$ is a global current on $Y'$ of the form  \eqref{na}.

%For the case when the analytic singularities are rational or real,
%see Remark~\ref{Qqas}.

\begin{remark}\label{bong}
Note that $\nu$ does not depend on the choice of Hermitian norm on $\La$.
One can in fact incorporate $a$ into the norm on $\La$, i.e., choose a
qas norm $\|\cdot\|^2$ on $\La$ (with empty degeneracy locus) so that
$\pi^*u=\log\|\sigma\|^2$.
\end{remark}

\begin{lma}\label{elbas2}
Assume that $\gamma$ is a product on $W$ as in \eqref{na} and $f$ is a non-trivial holomorphic function.
Then
$(\log|f|^2) \gamma$ has locally finite mass, and
$
dd^c( (\log|f|^2) \gamma)=[\dv f]\w \gamma.
$
If $u_j$ are smooth and decrease to $\log|f|^2$, then $u_j\gamma\to (\log|f|^2) \gamma$.
\end{lma}

\begin{proof}
The first statement follows from \cite[Theorem III.4.5]{Dem} since locally in $Y'$ all $b_j$ in $\gamma$ are
$b_j=b_j'+a_j$ where $b_j'$ are psh and bounded and $a_j$ are smooth.
The second statement follows from
 \cite[Proposition III.4.9]{Dem} and the third one from dominated convergence.
\end{proof}

\begin{prop}\label{kamaxel}
Let $u$ be a \qas\ function on $Y$ with degeneracy locus $Z$. If $\mu\in\QZ(Y)$, then
$u \1_{Y\setminus Z}\mu$  has order zero and locally finite mass and
\begin{equation}\label{klammerdef}
[dd^c u]\w \mu:= dd^c(u \1_{Y\setminus Z}\mu)
\end{equation}
is an element in $\QZ(Y)$.
\end{prop}

We interpret the right hand side of \eqref{klammerdef} as $0$ if $\mu$ has support in $Z$.
In particular, $[dd^c u]\w \mu=0$ for all $\mu$ if  $u\equiv -\infty$. % so that $Z=X$, then $[dd^c u]\w \mu=0$ for all $\mu$.

\begin{proof}
By linearity we can assume that $\mu=\tau_*\gamma$,  $\tau\colon W\to Y$. If the image of $\tau$ is contained in $Z$,
then $\1_{Y\setminus Z} \mu=0$. Otherwise $\tau^* u \not\equiv -\infty$ is \qas\ on $W$ by Lemma~\ref{lemur}. Then
$$
\1_{Y\setminus Z}\mu=\tau_*(\1_{W\setminus \tau^{-1}Z} \gamma)
=\tau_*\gamma=\mu
$$
since $\tau^{-1}Z$ has
positive codimension on $W$, cf.~\eqref{eq:BTprop1}.
In view of the discussion above we may assume, after possibly taking a modification of $W$,
that $\tau^*u=\log|\sigma|^2+a$ on $W$, cf.~\eqref{lemur2}, where $a$ is locally bounded.

Since $u$ is upper semi-continuous there is a sequence $u_j$ of smooth functions that decreases to $u$.
%%
%$$
%u_j \mu=\tau_*( \tau^*u_j \gamma).
%$$
By Lemma~\ref{elbas2}, $\tau^*u_j \gamma$ tends to the $\tau^*u\gamma$ that has locally finite mass. It follows that
$$
u\mu=\lim_j u_j \mu=\lim \tau_*( \tau^*u_j \gamma)=\tau_*(\tau^*u\gamma)
$$
is a locally finite measure. Moreover,
$dd^c( u \mu)=\tau_*( dd^c (\tau^* u \gamma))$
so it is enough to see that $dd^c (\tau^* u \gamma)$ is in $\QZ(W)$.
By \eqref{kontor} on $W$, and Lemma~\ref{elbas2},
\begin{equation}\label{baktankar}
dd^c (\tau^* u \gamma)=[\dv \sigma]\w \gamma+\nu\w\gamma.
\end{equation}
The last term in \eqref{baktankar} is of the form \eqref{na}.
Notice that $[\dv \sigma]$ is a sum of terms
$m_\ell[D_\ell]$, where $m_\ell$ are integers and $D_\ell$ are divisors in $W$.
In view of Lemma~\ref{elbas}~(ii), the first term on the right hand side of \eqref{baktankar} is therefore a finite
sum of terms
$m_\ell (i_{\ell})_*(i_\ell^* \gamma)$, where $i_\ell\colon D_\ell\to W$, and each of them is in $\QZ(W)$.
\end{proof}

Recursively we define
\begin{equation}%\label{boxer}
[dd^c u]^0\w \mu=\mu, \quad [dd^c u]^{k+1}\w \mu=[dd^c u]\w [dd^c u]^k\w \mu, \quad k=0,1,2, ....
 \end{equation}
 Let\footnote{If $u=\log|\sigma|^2$, then $M^u$ is denoted by $M^\sigma$ in, e.g., \cite{AESWY}. We hope this will not cause any confusion.}
 \begin{equation}\label{bajonett}
 \la dd^c u\ra^k\w \mu:=\1_{X\setminus Z} [dd^c u]^k \w\mu, \quad M^u_k\w \mu:= \1_Z [dd^c u]^k \w\mu.
 \end{equation}
 We have, cf.~\eqref{klammerdef},
 \begin{equation}\label{bajontva}
 [dd^c u]^{k+1}\w \mu =dd^c (u \la dd^c u\ra^{k}\w \mu).
 \end{equation}
 Let $M^u=M^u_0+M^u_1+\cdots$. Notice that if $\mu$ has support on $Z$, then $M^u_0\w \mu=\mu$.

\begin{ex}\label{fantomen}
With the notation in \eqref{baktankar} in the proof of Proposition~\ref{kamaxel} we have
\begin{equation}\label{fantomen1}
[dd^c\tau^*u]^k\w\gamma=M_k^{\tau^*u}\w\gamma+ \la dd^c \tau^*u\ra^k\w\gamma=
[\dv \sigma]\w \nu^{k-1}\w\gamma+\nu^k\w\gamma.
\end{equation}
\end{ex}

\begin{prop}\label{dragos}
If $\tau\colon Y'\to Y$ is a proper mapping and $\mu\in\QZ(Y')$, then
\begin{multline}\label{oxel}
 [dd^c u]^{k}\w \tau_*\mu= \tau_*\big([dd^c \tau^* u]^k\w \mu\big), \quad  \la dd^c u\ra^{k}\w \tau_*\mu= \tau_*\big(\la dd^c \tau^* u\ra^k \w \mu\big), \\
 \quad M_k^u\w \tau_* \mu= \tau_*\big( M_k^{\tau^*u}\w \mu\big), \quad k=0,1,2,\ldots.
\end{multline}
\end{prop}

\begin{proof}
The case $k=0$ is trivial, and by induction
it is enough to consider $k=1$. Then the proposition follows from the equalities $\1_Z \tau_* \mu=\tau_*(\1_{\tau^{-1}Z} \mu)$ and %thus
$$
dd^c( u_j \1_{Y\setminus Z} \tau_*\mu)=\tau_* \big(dd^c (\tau^*u_j \1_{Y'\setminus \tau^{-1}Z} \mu)\big),
$$
where $u_j$ is a sequence of smooth functions as in the proof of Proposition~\ref{kamaxel}.
\end{proof}

\begin{remark} \label{Qqas}
We claim that if $u$ is $\mathbb{Q}$-\qas/$\mathbb{R}$-\qas\, then  there is a modification
$\pi \colon Y' \to Y$ and a section $\sigma$ of a line bundle $\L\to Y'$ such that locally on $Y'$,
 $\pi^* u = c\log|\sigma|^2+a$, where $a$ is qpsh and $c>0$ is rational/real.  Notice that
 $c$ cannot be globally constant in general. For instance, let $u$ be a potential to
 $\sum_k (1/k)[z=k]$ in the complex plane.

To see the claim, first recall that the normalization $p\colon U'\to U$ of the blowup of $U$ along the ideal $\J$ generated by
$f$ is the same as the blow up of the ideal $\J^k$ for any $k\ge 1$.
If $f$ and $\hat f$ occur in two local representations \eqref{totem1} then
$c'\log|f'|=c\log|f|+\O(1)$.   Therefore $|f'|^{k'}\sim |f|^k$ for some integers $k,k'\ge 1$. This is
 obvious if $u$ is  $\mathbb{Q}$-\qas\ so that $c,c'\in \mathbb Q$.  In the real case this is clear if $f$ and $f'$ are functions, and one is reduced to this case in a suitable modification where $f$ and $f'$ are principal. %reduced to the case of functions afterand follows
 Thus $Y'$ is well-defined locally and hence globally. Now the claim follows if
$\sigma$ is a global section that defines the exceptional divisor.  %then the claim follows. %$\sigma$ of  $\L\to Y'$ such that  Then one proceeds as
\end{remark}

\subsection{Action of Segre and Chern forms associated with a smooth  metric}
Notice that if $L\to X$ is a line bundle with a smooth Hermitian metric $e^\psi$,
then $c_1(L,e^\psi)$ defines an endomorphism $\c_1(L,e^\psi)$ on currents acting as multiplication
by the smooth form $c_1(L,e^\psi)$.
Moreover, if
$\mu=\tau_*\gamma$, then $c_1(L,e^\psi)\w \tau_*\gamma=\tau_*(c_1(\tau^*L,\tau^*e^{\psi})\w\gamma)$
and hence $\c_1(L,e^\psi)$ yields an endomorphism on $\QZ(X)$. We also have the endomorphism
$\s_1(L,e^\psi)=-\c_1(L,e^\psi)$.

Let $E\to X$ be a Hermitian holomorphic vector bundle. Since $p\colon \Pk(E)\to X$ is a submersion,
there is a well-defined morphism $p^*$ mapping currents on $X$ to currents on $\Pk(E)$ such that
$p_*\xi \w \mu=p_*(\xi\w p^*\mu)$
for smooth forms $\xi$ on $\Pk(E)$.

\begin{lma}\label{tomte1}
We have
$$
p^*\colon \QZ(X)\to \QZ(\Pk(E)).
$$
\end{lma}

\begin{proof}
By functoriality
we have the commutative diagram
\begin{equation}\label{dia23}
\begin{array}[c]{cccc}
 \Pk(\tau^*E) &  \stackrel{\tau'}{\longrightarrow}  & \Pk(E) \\
\downarrow \scriptstyle{p'} & &  \downarrow \scriptstyle{p}  \\
W & \stackrel{\tau}{\longrightarrow}  & X.
\end{array}
\end{equation}
If  $\mu=\tau_*\gamma$, thus $p^*\mu=\tau'_* (p')^*\gamma$, which is an element in $\QZ(\Pk(E))$.
 \end{proof}

Since $s(E,h)=p_*s(L_E,e^\psi)$,
for any current $\mu$ on $X$,
\begin{equation}\label{toka}
s(E,h)\w \mu= p_* s(L_E,e^\psi)\w \mu=p_*(s(L_E,e^\psi)\w p^*\mu).
\end{equation}
It follows from the lemma that \eqref{toka} defines endomorphisms
\begin{equation}\label{kraka}
\s_k(E,h)\colon \QZ(X)\to \QZ(X), \quad k=0,1,\ldots,
\end{equation}
where
$\s_k(E,h)\w\mu:= p_*(s_1(L_E,e^\psi)^{k+r-1}\w p^*\mu)$.
In the same way we have endomorphisms
\begin{equation}\label{kraka1}
\c_k(E,h)\colon \QZ(X)\to \QZ(X), \quad k=0,1, \ldots
\end{equation}
defined by suitable compositions, cf.~\eqref{cformel}.
Since these operators are just multiplication by $s_k(E,h)$ and $c_k(E,h)$, in particular %we have that
$\c_0(E,h)=\s_0(E,h)=I$, the identity on $\QZ(X)$.

\subsection{The submodule $\QZ_0(Y)$}

Let $0\to S\to E\to Q\to 0$ be a short exact sequence of holomorphic vector bundles over a complex manifold $W$ with the metrics induced by a smooth metric
on $E$. Then, by \cite{BC}, for each component
$\beta=(c(E)-c(S)c(Q))_k$ there is smooth form $v$ on $W$ such that
$dd^c v=\beta$.
%We say that such a $\beta$ is a BC-form.
It follows from \eqref{kraka1} that $\beta$ induces an endomorphism
$\QZ(W)\to \QZ(W)$, \ $\mu\mapsto \beta\w \mu$.

Let $L$ be a holomorphic line bundle over $W$ with two different metrics $e^b$ and $e^{b'}$,
where $b'$ is assumed to be smooth, and $b$ is qpsh and locally bounded. Then $v=b-b'$ is a
locally bounded qpsh function. %and thus $e^v$ may be considered as a metric on the trivial line bundle.
Let $\beta:=dd^c b'-dd^c b=dd^c v$. Again we have a mapping $\QZ(W)\to \QZ(W)$, \ $\mu\mapsto \beta\w \mu$.
We say that any of these two forms $\beta$ is a {\it beta form}.

\begin{df}
We let $\QZ_0(Y)$ be the submodule of $\QZ(Y)$ generated by all currents of the form
$\tau_*\mu'$, $\tau\colon W\to Y$, where $\mu'$ is in the image of one of the mappings  above.
\end{df}

Thus  $\QZ_0(Y)$ is generated by all $\tau_*(\beta\w\gamma)$
where $\beta$ is a beta form and $\gamma$ is a product as in  \eqref{na}.
For each such form there is a locally bounded function $v$ such that
\begin{equation}\label{parasoll}
dd^c \tau_*(v\gamma)=\tau_*(\beta\w\gamma).
\end{equation}

 \begin{remark}
 Since the pullback of a beta form is a beta form, the module $\QZ_0(Y)$ is actually generated by all currents
 $\tau_*\gamma$, where $\gamma$ is as in \eqref{na} but with at least one factor $dd^c b_j$
 'replaced' by a beta form.
\end{remark}

If $\tau\colon Y'\to Y$ is a proper mapping, then $\tau_*\colon \QZ_0(Y')\to \QZ_0(Y)$.
In view of \eqref{eq:restrGGZ}, for any analytic subspace $V\subset Y$,
\begin{equation}\label{stare}
\1_V \colon \QZ_0(Y)\to \QZ_0(Y).
\end{equation}

%We have the following dimension principle.
\begin{prop} \label{elefant}
If  $\mu\in \QZ_0(Y)$ has bidegree $(k,k)$ and support on the analytic subspace $V \subset Y$,
then $\mu=0$ if $k-1<\codim V$.
\end{prop}

In particular, if $\mu\in \QZ_0(Y)$ has support at a point, then $\mu=0$.

\begin{proof}
It follows from %In fact
%\begin{remark} \label{elefant}
%If $\mu\in \QZ_0(Y)$ has support on the analytic subspace $V \subset Y$,
Proposition~\ref{tax} and \eqref{parasoll} that there is
a current $w=i_*w'$, where $i\colon V\to X$,
such that $dd^c w=\mu$.  If $\mu$ has bidegree $(k,k)$, then $w$ has bidegree $(k-1,k-1)$.
Thus $w'$ has bidegree $(k-1-\codim V,k-1-\codim V)$ and hence it vanishes if $k-1<\codim V$.
\end{proof}

\begin{prop}\label{ekorre}
(i) If $u$ is a \qas\ function on $Y$, then
$$
[dd^c u]^k\colon \QZ_0(Y)\to \QZ_0(Y), \quad k=0,1,2,\ldots.
$$
(ii)
If $u'$ is \qas\ and $u'=u+\Ok(1)$ locally on $Y$, then
$$
[dd^c u]^k-[dd^c u']^k \colon \QZ(Y)\to \QZ_0(Y), \quad k=0,1,2,\ldots.
$$
(iii) The mapping $p^*$ in Lemma~\ref{tomte1} maps $\QZ_0(X)\to \QZ_0(\Pk(E))$. \newline
\end{prop}

\begin{proof}
Recall that if $\mu$ has support in $Z$, then
$[dd^c u]\w \mu=0$. Otherwise, assume that $\mu=\tau_* (\beta\w \gamma)$ where $\gamma$ is product as in \eqref{na}
and $\beta$ is a beta form. As in the proof of Proposition~\ref{kamaxel} we can therefore assume, cf.~\eqref{baktankar},
that
\begin{equation}\label{skunk}
[dd^c u]\w \mu =\tau_*\big([\dv\sigma]\w \beta\w\gamma+\nu\w\beta\w\gamma\big).
\end{equation}
If $i\colon D\to W$, then $i^*(\beta\w \gamma)=i^*\beta\w i^*\gamma$ and thus
$[\dv\sigma]\w \beta\w \gamma$ is in $\QZ_0(W)$.
%that $[dd^c u]\w\mu$ is the push-forward of such forms, since if $i\colon D\to W$, then $i^*(\beta\w \gamma)$ is such a form.
Thus (i) holds for $k=1$ and hence for a general $k$ by induction.

%\smallskip
By (i) and a simple induction over $k$ it is enough to prove (ii) for $k=1$. Let us assume that $\mu=\tau_*\gamma$, $\tau\colon W\to Y$.
After a suitable further modification we can assume that $\tau^*u=\log|\sigma|^2+a$ and $\tau^*u'=\log|\sigma'|^2+a'$ in $W$
where $a$ and $a'$ are locally bounded.
The assumption in (ii) implies that $[\dv\sigma']=[\dv\sigma]$.
If $a''$ is some smooth metric on the dual of the line bundle corresponding to this divisor, then
$dd^c(a-a')=dd^c(a-a'')-dd^c(a'-a'')=\beta-\beta'$ where both $\beta$ and $\beta'$ are beta forms on $W$.
Therefore,
$$
([dd^c u']-[dd^c u])\w\mu= \tau_*(dd^c \tau^*u \w \gamma) -\tau_*(dd^c \tau^*u' \w \gamma)=
\tau_*(\beta\w\gamma)-\tau_*(\beta'\w\gamma)
$$
which is in $\QZ_0(Y)$.

%\smallskip
Part (iii) follows from the proof of Lemma~\ref{tomte1}. In fact, if $\mu$ is in $\QZ_0(X)$, then we can assume that
 $\mu=\tau_*(\beta\w\gamma)$ where $\gamma$ is of the form \eqref{na} and $\beta$ is a beta form.
Then $p^*\mu=\tau'_* (p')^*(\beta\w\gamma)$ is in $\QZ_0(\Pk(E))$.
\end{proof}

 \begin{df}
 We define the quotient module $\QB(Y):=\QZ(Y)/\QZ_0(Y)$.
 \end{df}

The definition is analogous with the definition of the quotient module $\Ba(Y)$ of $\GZ(Y)$, see \cite{AESWY}, and there is a
natural mapping
$\mathcal \Ba(Y)\to \QB(Y)$.

\subsection{Multiplicities of quasi-cycles}\label{multip}

Let $\xi$ be a tuple of holomorphic functions that defines the maximal ideal at $x\in Y$.
If $\mu\in \QZ(Y)$ and $u=\log|\xi|^2$, then
$M^u\w \mu$
is in $\QZ(Y)$ and has support at $x$, so it is $\alpha[x]$ for some integer $\alpha$, cf.~Example~\ref{pudel}.
By Proposition~\ref{ekorre}~(ii) and Proposition~\ref{elefant},  $\alpha$
is independent of the choice of tuple $\xi$ defining the maximal ideal.
If $\mu$ has bidegree $(k,k)$, for degree reasons, we have that
$M^u\w\mu=M^u_{n-k}\w \mu$.

\begin{df}
Let $\mu\in \QZ(Y)$ be a $(k,k)$-current and $x\in Y$. We define
the multiplicity of $\mu$ at $x$, $\mult_x\mu$, as the integer
\begin{equation}\label{kastvind}
\mult_x \mu [x] =\int M^u\w \mu = \int \1_{\{x\}}(dd^c\log|\xi|^2)^{n-k} \w \mu.
\end{equation}
\end{df}

If $\mu$ is a positive closed current, \eqref{kastvind} means
that $\mult_x\mu$ is the Lelong number of $\mu$ at $x$.

If $\mu\in \QZ_0(Y)$,  then $M^u\w\mu\in \QZ_0(Y)$ by Proposition~\ref{ekorre}~(i),
and hence $\mult_x\mu=0$, cf.~Proposition~\ref{elefant}. That is,

\begin{prop}\label{gorilla}
The multiplicity $\mult_x\mu$ only depends on the class of $\mu$ in $\QB(Y)$.
\end{prop}

\begin{thm}\label{superman}
Assume that $\mu\in \QZ(Y)$.
If $\mu$ has bidegree $(k,k)$, then there is a unique decomposition
\begin{equation}\label{siu}
\mu=\mu_F+\mu_M,
\end{equation}
where $\mu_F$ is the Lelong current of an analytic cycle on $Y$ of codimension $k$, and the set where $\mult_x \mu_M$
does not vanish is contained in an analytic set of codimension $\ge k+1$.
\end{thm}

If $\mu'$ defines the same class in $\QB(Y)$, then, cf.~Proposition~\ref{gorilla}, it has the same multiplicities. It follows
from the theorem that
$\mu'_F=\mu_F$. Clearly $\mu'_M$ and $\mu_M$ may differ, but they still have the same multiplicities at
each point. We say that $\mu_F$ is the {\it fixed} part and $\mu_M$ is the {\it moving} part of the class of $\mu$ in $\QB(Y)$.

\begin{proof}
Since we know that $\mult_x\mu$ is integer valued, \eqref{siu} follows directly from the
Siu decomposition in case $\mu$ is a positive current.
As in the case with $\mu\in\GZ(Y)$, if $\mu=\tau_*\gamma$, and
$\tau=i\circ\tau'$, where $\tau'\colon W\to Z$ is surjective, then
$\mult_x \tau'_*\gamma=\mult_{i(x)}\tau_*\gamma$. Moreover,
\cite[Lemma~7.3]{Asommar} holds, with the same proof, for $\mu=\tau_*\gamma$ if
$\gamma$ is as in \eqref{na}. In fact, the crucial point is that $M^{\tau^*\xi}\w \gamma=0$
if $M^{\tau^*\xi}=M^{\tau^*\xi}\w 1=0$. To see this we can assume that $\xi$ is a tuple of functions
(i.e., a section of a vector bundle with trivial metric) in a \nbh
of $x$. Then $M^{\tau^*\xi}$ is a positive current.
If $b_\ell$ is locally of the form $u_\ell + g_\ell$, where $u_\ell$ is psh, and $g_\ell$ is smooth,
we may take a sequence $b_\ell^j = u_\ell^j + g_\ell$, where each $u_\ell^j$ is psh, and
$u_\ell^j$ decreases to $u_\ell$. Then, by \cite[Proposition III.4.9]{Dem},
$M^{\tau^*\xi}\w \gamma_j\to M^{\tau^*\xi}\w \gamma$. If $M^{\tau^*\xi}=0$, thus
$M^{\tau^*\xi}\w \gamma_j=0$ and hence $M^{\tau^*\xi}\w \gamma=0$.
Now Theorem~\ref{superman} follows precisely as
\cite[Theorem~7.1]{Asommar}.
\end{proof}

\begin{remark} \label{multQ}
It follows that if $\mu\in\QZ_\Q(Y)$, then $\mult_x\mu$ is a rational number.
\end{remark}

\begin{ex}
Assume that $\mu\in\QZ(Y)$ has bidegree $(k,k)$ and support on an analytic subspace
$i\colon V\to Y$ of pure codimension $q$. In view of Proposition~\ref{tax},  $\mu=i_*\tau$ for
a $\tau\in \QZ(V)$ of bidegree $(k-q,k-q)$.
Thus $\mu=0$ if $k<q$. If $k=q$, then $\tau$ is a locally constant function which is integer-valued since its multiplicities
are integers. It follows that $\mu$ is a cycle, i.e.,  in $\Zc (Y)$.
If $\mu$ is in $\QZ_0(Y)$, then it vanishes by Proposition~\ref{elefant}. Thus
$\Zc (Y)$ is a submodule of $\QB(Y)$.
\end{ex}

\section{Regularization of \qas\ metrics on a line bundle}

Let $L\to Y$ be a holomorphic vector bundle, let $e^\psi$ be a qas metric on $L$, % with degeneracy locus $Z$,
and let $e^{\psi_0}$ be a smooth reference metric. Then, cf.~Example~\ref{plaskdamm},
the global function $\psi-\psi_0$ is \qas\ on $Y$.
In view of Proposition~\ref{kamaxel} the following definition makes sense.
\begin{df}
Consider a qas metric $e^\psi$ on $L \to X$, a smooth reference metric $e^{\psi_0}$, and  $\omega=dd^c\psi_0$.
We introduce the endomorphism
\begin{equation}\label{harry}
[dd^c\psi]_\omega\w\mu=[dd^c(\psi-\psi_0)]\w \mu+\omega\w\mu
\end{equation}
on $\QZ(Y)$.
\end{df}

Recall that if $\psi\equiv-\infty$, then $[dd^c(\psi-\psi_0)]\w\mu=0$, so that $[dd^c\psi]_\omega=\omega\w\mu$.

\begin{prop}\label{ekorre2}
(i) If $e^\psi$ is a \qas\ metric on $L\to Y$, then $[dd^c\psi]^k_\omega\colon \QZ_0(Y)\to\QZ_0(Y)$.

\smallskip
\noindent
(ii) If $e^{\psi'}$ is another \qas\ metric and $\psi'=\psi+\Ok(1)$ then
$[dd^c\psi']^k_\omega-[dd^c\psi]^k_\omega\colon \QZ(Y)\to\QZ_0(Y)$.

\smallskip
\noindent
(iii) If $\tau\colon Y'\to Y$ is a proper mapping and $\mu=\tau_*\mu'$, then
 \begin{equation}\label{harry2}
[dd^c\psi]_\omega\w\mu=\tau_*\big( [dd^c \tau^*\psi]_{\tau^*\omega}\w \mu'\big).
\end{equation}
\end{prop}

\begin{proof}
%If $\psi=-\infty$ all the statements are trivial. Otherwise,
The case $k=1$ of (i) follows directly from \eqref{harry}, Proposition~\ref{ekorre}~(i),
and the fact that multiplication by $\omega$ maps $\QZ_0(Y)$ into itself.
The general
case follows by trivial induction.
Part (ii) follows direct from Proposition~\ref{ekorre}~(ii) applied to $u'=\psi'-\psi_0$
and $u=\psi-\psi_0$.
Part (iii) follows from Proposition~\ref{dragos}. %Notice that if \tau^*\psi$ is $-\infty$, contained in the
\end{proof}

It is sometimes useful to express $[dd^c\psi]_\omega$ %ometimes useful to express it
in terms of $[dd^c\psi]$ which is well-defined even when  $\psi$  is a metric. In fact, \eqref{klammerdef} is local and not affected if
$u$ is replaced by $u+h$, where $h$ is pluriharmonic.
Let $Z$ be the degeneracy locus of $e^\psi$.
If $\mu=\1_{Y\setminus Z}\mu$, then \eqref{harry} is
$
[dd^c(\psi-\psi_0)]\w\mu+\omega\w\mu=[dd^c\psi] \w \mu,
$
and if $\mu=\1_Z\mu$, then it is just $\omega\w\mu$.
Therefore, %It is therefore natural to write
\begin{equation}\label{harry1}
[dd^c\psi]_\omega\w\mu= [dd^c\psi]\w\1_{Y\setminus Z}\mu+\omega\w\1_Z \mu;
\end{equation}
this is precisely the definition of $[dd^c\psi]_\omega$ in \cite{LRSW}.
Besides $[dd^c\psi]^k$ we
have endomorphisms
$$
 \la dd^c \psi\ra^k\w \mu=\1_{Y\setminus Z} [dd^c\psi]^k\w\mu,   \quad  M^\psi_k\w\mu =\1_{Z} [dd^c\psi]^k\w\mu;
 $$
on $\QZ(Y)$.
It is readily verified by induction that
\begin{equation}\label{kusin2}
[dd^c \psi]^k_{\omega}\w\mu= \la dd^c \psi\ra^k\w \mu+\sum_{\ell=0}^k \omega^{k-\ell}\w M^{\psi}_\ell\w\mu, \quad k=0,1,2,\ldots.
 \end{equation}

\begin{thm}\label{kusin}
Let $\chi_j$ be a sequence of smooth bounded from below convex increasing functions on $\R$, which decrease
to the identity function as $j \to \infty$.
Let $e^\psi$ be a \qas\ metric on $L\to Y$, $\psi_0$ a smooth metric, and let
$\psi_j=\chi_j\circ(\psi-\psi_0) + \psi_0$.
If $\omega=dd^c\psi_0$, then for each $\mu\in\QZ(Y)$,
\begin{equation}\label{troll3}
(dd^c \psi_j )^k\w \mu \to [dd^c \psi]^k_{\omega}\w\mu, \quad j\to \infty,  \quad k=0,1,2,\ldots.
\end{equation}
\end{thm}

Since $\psi-\psi_0$ is a global function $\psi_j$ are well-defined.
The smoothness assumption on $\chi_j$ is not necessary, but for simplicity we restrict to this case, since we
will only apply the theorem to $\chi_\e(t)=\log(e^t+\e)$.

This theorem is known when $\mu$ is $1$ and with $[dd^c\psi]^k_\omega$ expressed as in  \eqref{kusin2}.
In \cite{AWW} it is proved for a more general class of (q)psh functions. Here we will rely on the presentation in \cite{Bl} and
show that the proof given there works even when acting on a quasi-cycle $\mu$.

\begin{proof}
First assume that $\psi\equiv -\infty$. Then $\psi_j=\chi_j(-\infty)+\psi_0$ so that
$(dd^c \psi_j)^k\w\mu=\omega^k\w\mu=[dd^c\psi]^k_\omega\w\mu$. Let us thus assume that
$\psi\not\equiv -\infty$.
Since the statement is local, we can restrict to an open subset where we have a global frame element, and hence
that $\psi$ and $\psi_0$ are fixed functions. For simplicity we let $Y$ denote this open set.

We begin with the case when $\psi_0=0$.  %
If $\mu$ has support
on $Z$, then both sides of \eqref{troll3} are equal to $\omega^k \w \mu$.
Thus we can assume that $\mu=\tau^*\gamma$, where $\tau\colon W\to Y$, $\gamma$ is a product as in \eqref{na}
and $\tau^{-1} Z$ has positive codimension in $W$.
After possibly another modification we may assume that
$\tau^*\psi= \log|\sigma|^2 + a$,
where $\sigma$ is a section of a Hermitian line bundle $\L\to W$ and $a$ is qpsh and locally bounded.
We claim that %
\begin{equation}\label{viking}
(dd^c \chi_j\circ \tau^*\psi)^k \w \gamma\to [dd^c\tau^*\psi]^k\w \gamma.
\end{equation}
Following \cite{Bl} there are unique convex increasing functions $\gamma_j$ such that $\gamma_j(-1)=\chi_j(-1)$ and
 $
\gamma'_j= (\chi'_j)^k.
$
It follows that also $\gamma_j(t)$ decreases to $t$. A direct computation yields
$$
(dd^c (\chi_j\circ\tau^*\psi))^k\w \gamma=dd^c( (\gamma_j\circ\tau^*\psi) \nu^{k-1}\w\gamma),
$$
where $\nu = dd^c a - c_1(\L)$.
Since $\gamma_j\circ\tau^*\psi \to \tau^*\psi$ in $L^1$ with respect to the (total variation of the) measure $\nu^{k-1}\w\gamma$, we see that
by definition,
$$
(dd^c \chi_j\circ\tau^*\psi)^k\w \gamma\to dd^c (\tau^*\psi \nu^{k-1}\w\gamma)=([\dv\sigma]+\nu)\w \nu^{k-1}\w\gamma =
[dd^c \tau^*\psi]^k\w\gamma;
$$
Here the last equality follows from Example~\ref{fantomen} and $\nu=dd^c a-c_1(\L)$.
Thus \eqref{viking} holds. Applying $\tau_*$ and using
Proposition~\ref{dragos}, now \eqref{troll3} follows for $\psi_0=0$.
Let us now assume that $\psi_0$ is an arbitrary smooth function.
Using the case $\psi_0=0$, %we have
\begin{multline}\label{trull}
(dd^c \psi_j)^k\w\mu=\sum_{\ell=0}^k {k \choose\ell} \omega^{k-\ell} \w(dd^c \chi_j\circ(\psi-\psi_0))^\ell\w\mu\to \\
\sum_{\ell=0}^k {k\choose \ell} \omega^{k-\ell} \w [dd^c(\psi-\psi_0)]^\ell \w\mu = ([dd^c(\psi-\psi_0)]+\omega)^{k}\w\mu,
\end{multline}
where the last power is composition of the endomorphism $\mu\mapsto ([dd^c(\psi-\psi_0)]+\omega)\w\mu$ on $\QZ(Y)$.
Now \eqref{troll3} follows from \eqref{harry} and \eqref{trull}.
\end{proof}

\section{Segre forms and endomorphisms of \qas\ bundles}\label{segreendo}

Assume that $(E,h)$ is a \qas\ bundle over $X$. If $e^\psi$ denotes the induced
metric on $L_E$, then $\psi=\log |\alpha|^2_h$ is \qas. Thus for any local frame
$a$ of $L_E$  the function $\log|\alpha a|^2_h$ is \qas\  on $\P(E)$.
On the set where $\alpha_j \neq 0$ we may take $a=1/\alpha_j$.

Let $h_0$ be a smooth reference metric on $E$, let $e^{\psi_0}$ denote
the induced metric on $L_E$, and $\omega = dd^c \psi_0$.
We thus have that
$$
    s(L,e^{\psi_0}) = \sum_{\ell=0}^\infty\omega^\ell=\frac{1}{1-\omega}.
$$
If $Z'\subset\Pk(E)$ denotes the set where $\psi$ is not locally bounded,
then outside of $Z'$ we have that $s_1(L_E,e^\psi)=dd^c\psi$.
Recall, cf.~\eqref{klammerdef} and \eqref{harry1}, that
\begin{equation}\label{smart1}
[dd^c \psi]_\omega\w \mu= dd^c (\psi \1_{X\setminus Z'}\mu)+\omega\w \1_{Z'} \mu
\end{equation}
defines an endomorphism on $\QZ(\Pk(E))$.
%In view of
%
In view of Lemma~\ref{tomte1} and Proposition~\ref{ekorre2}~(i)
the following definition makes sense.

\begin{df}\label{kontur}
Given a \qas\ bundle $(E,h)$ over $X$, and a smooth metric $h_0$, we have the Segre endomorphisms
 $\QZ(X)\to \QZ(X)$ and $\QZ_0(X)\to \QZ_0(X)$ of bidegree $(k,k)$,
%defined by
\begin{equation}\label{torrfoder}
\s_k(E,h,h_0)\w \mu=%p_*([dd^c\psi]_\omega^{k+r-1}) \w \mu=
p_*( [dd^c\psi]_\omega^{k+r-1}\w p^*\mu).
\end{equation}
We let
$$
\s(E,h,h_0) := \sum_k \s_k(E,h,h_0).
$$
 \end{df}

%If $h$ is smooth, $[dd^c \psi]_\omega$ is just multiplication with the smooth
%form $dd^c \psi$, so $\s_k(E,h,h_0)$ is just multiplication by the smooth form
%$s_k(E,h)$, see \eqref{toka}.

%\smallskip

We define the Segre form
$s(E,h,h_0):=\s(E,h,h_0) 1$ on $X$, i.e.,
\begin{equation}\label{segre0}
s(E,h,h_0):=p_* \sum_{k=0}^\infty  [dd^c\psi]^k_\omega.
\end{equation}
Note that if $L$ is a line bundle, then $p_L \colon \P(L) \to X$ is the identity.
Hence,
\begin{equation}\label{segreL}
s(E,h,h_0):= p_* s(L_E,e^\psi,e^{\psi_0}).
\end{equation}

\begin{remark} \label{remLRSWcomparison}
The definition of $s(E,h,h_0)$ here coincides with the definition in
\cite{LRSW}.  However, in that article
is used the convention that $\Pk(E)_x$ denotes the hyperplanes in $E_x$ so one has to make
some necessary adaptions. Moreover, the definition is made on the associated (Griffiths positive) dual bundle
$(E^*,h^*)$, see, e.g., \cite[Lemma~3.1]{LRRS}. As in the smooth case one have the relations
$$
   s_k(E,h,h_0) = (-1)^k s_k(E^*,h^*,h_0^*),
$$
which connect the definitions in this article and in  \cite{LRSW}.
%where the right-hand side is the Segre form defined in \cite{LRSW} (but with $E^*$ denoted by $E$ there) and the left-hand
%side is the Segre form defined in this article.
\end{remark}

By \eqref{kusin2}, with $Y=\Pk(E)$ and $Z=Z'$, we have
\begin{multline}\label{rakning}
\sum_{k=0}^\infty	 [dd^c \psi]^k_\omega \w\mu
	 =\sum_{k=0}^\infty \la dd^c \psi\ra^k	\w \mu
+\sum_{\ell=0}^\infty \sum_{k=\ell}^\infty M^\psi_\ell\w (dd^c\psi_0)^{k-\ell}\w\mu=\\
\left(\sum_{k=0}^\infty \la dd^c \psi\ra^k + s(L_E,e^{\psi_0})\w M^\psi\right)\w\mu.
\end{multline}
Thus
\begin{equation}\label{segre}
\s(E,h,h_0)\w \mu=p_*\big(\sum_{\ell=0}^\infty  \la dd^c\psi\ra^\ell\w p^*\mu + s(L_E,e^{\psi_0})\w M^\psi \w p^*\mu\big).
\end{equation}
Since the support of $M^\psi \w p^*\mu$ is contained in $p^{-1}Z$,
\begin{equation}\label{storlom}
 \1_{X\setminus Z} \s(E,h,h_0)\w \mu=p_*\big(\1_{\Pk(E)\setminus p^{-1}Z}\sum_{\ell=0}^\infty \la dd^c\psi\ra^\ell\w p^*\mu\big).
\end{equation}
It is independent of choice of reference metric $h_0$ and we denote it by $\s'(E,h)\w\mu$.
%We therefore have the decomposition
%\begin{equation}\label{segre}
%\s(E,h,h_0)\w\mu=\sum_{\ell=0}^\infty p_* (\la dd^c\psi\ra^\ell\w p^*\mu) + p_*\big(s(L_E,e^{\psi_0})\w M^\psi \w p^*\mu\big)=
%
%
%\end{equation}
%
%\small-skip
Letting
$$
M^{E,h,h_0} \w \mu:=\1_{Z} \s_k(E,h,h_0) \w \mu
$$
we have the decomposition %y decompose $\s_k(E,h,h_0) \w \mu$ as
\begin{equation}\label{skatbo}
\s(E,h,h_0)\w\mu=    \s_k'(E,h)\w \mu + M^{E,h,h_0} \w \mu.
\end{equation}
%where the last term is  $\1_{Z} \s_k(E,h,h_0) \w \mu$.
Note that by \eqref{eq:restrGGZ}, both terms on the right-hand side belong to $\QZ(X)$.

\begin{ex}\label{favvis0}
Let $(E,h_0)$ and $(F,h_F)$ be Hermitian vector bundles over $X$,
and let $g\colon E\to F$ be a holomorphic morphism.
Let $h$ be the singular metric on $E$ defined by
$|\alpha|_h := |g\alpha|_{h_F}$.
Recall the representation $\psi=\log|\alpha|_h^2/|\alpha_j|^2$ of the induced metric
on $L_E$, \eqref{phiform}.  In \cite{Asommar}, the corresponding representation
is denoted by $\log|\alpha|^2_\circ$. It follows from \eqref{segre},
with $\mu=1$, that the definition given in \cite[Example~9.2]{Asommar} of $s(E)$ coincides with
\eqref{segre0}.
\end{ex}

\subsection{The case when $h$ is  qas and non-degenerate}

We need the following technical result.

\begin{prop} \label{tunika}
Let $\psi(x,y)$ be a  bounded qpsh function on $U \subset \C^n_x \times \C^m_y$,
such that $\psi$ is smooth in $y$ for each fixed $x$,
let $v(x,y)$ be a smooth psh function,
and let $\rho_\e(x)=\rho(x/\e)/\e^n$ be a family of radial mollifiers with respect to the variable $x$.
Define
\begin{equation} \label{eq:locBddQasApproximation}
    \psi_\e = \log(e^\psi*_x\rho_\e+A\e e^v),
\end{equation}
where $\psi *_x\rho_\e$ denotes the convolution with respect to $x$.
Then for any $V \subset\subset U$, $\psi_\e$ is smooth on $V$ for small $\e>0$,
% small enough,
and if $A \gg 1$, then $\psi_\e$ decreases to $\psi$ when $\e\to 0$.
Moreover, %if $\gamma$ is of the form \eqref{na}, then
\begin{equation} \label{tomte}
    (dd^c \psi_\e)^k\w \mu \to (dd^c \psi)^k\w \mu, \quad \mu\in \QZ(U).
\end{equation}
% is smooth, and
%$dd^c g \geq -B \omega$, where $\omega=dd^c (|x|^2+|y|^2)$, then
%$dd^c \psi_\e \geq - B \omega$.
\end{prop}

Since $\psi$ is locally bounded we write $(dd^c \psi)^k\w \mu$ rather than $[dd^c \psi]^k\w \mu$.

\begin{proof}
We can choose a constant $B>0$ such that
$
\psi+g
$
is psh on $V$ if $g= B(|x|^2+|y|^2)$.
For $\e>0$,
$$
\psi_\e=\log\big((e^\psi *_x \rho_\e) e^g+e^{v+\log(\e A)+g}\big) -g.
$$
We claim that the first term $\psi'$ on the right is psh. To see this first assume that $\psi$ is smooth. % ans consider a fixed $\e>0$.
If we approximate the convolution with a Riemann sum we get
 \begin{equation}\label{tiktok}
 \log\big(\sum_k e^{\psi(\cdot-x_k,\cdot)+c_k +g}+e^{v+\log(\e A)+g}\big).
\end{equation}
Since $dd^c g$ is translation invariant each $\psi(\cdot-x_k,\cdot)+g$ is psh.  Also $v+g$ is psh so the function in \eqref{tiktok} is psh. Since such functions tend uniformly to $\psi'$ we conclude that
$\psi'$ is psh. If $\psi$ is not smooth we approximate $u=\psi+g$ by a decreasing sequence $u_j$ of smooth psh functions and let $\psi_j=u_j-g$. By the same argument,
each $\psi'_j$ is psh, and since $\psi_j$ decreases to $\psi$, $\psi'$ is psh. Thus the claim is proved.

Note that $\psi_\e$ is smooth as a function of $x$ and $y$.
Clearly $\psi_\e\to \psi$ when $\e\to 0$. We now prove that $\psi_\e$ decreases with $\e$ on $V$ if $A$ is large enough.
To see this, we may fix $y$ and consider everything as a function of $x$, as long as the choice of $A$ is
uniform in $y$.
We write $\psi=u+g$, where $u$ is psh and
$g$ is smooth on $X$.  Then
$$
e^{\psi_\epsilon}(x) = e^{g(x)}  (e^u \ast \rho_\epsilon)(x) +
\int (e^{g(x-x')}-e^{g(x)}) \rho_\epsilon(x') dV(x') + A\epsilon e^v.
$$
Since $e^u$ is psh, the first term of the right hand side decreases when $\epsilon$ decreases.
To see that the whole expression is decreasing when $\epsilon$ decreases it
thus suffices to prove that the middle term is differentiable in $\epsilon$,
with a derivative less than $A e^v$.
In fact, $(d/d\epsilon) \rho_\epsilon(x') = \Ok(1/\epsilon^{n+1})$,
and has support where $|x'| \leq \epsilon$.
Since $|e^{g(x-x')}-e^{g(x)}| \leq C\epsilon$ for $|x'| \leq \epsilon$
(with $C$ depending on the $C^1$-norm of $e^g$),
the derivative of the middle term is bounded by a constant that is smaller than
$A e^v$ on $V$ if $A$ is chosen large enough. This choice can also be made uniformly in $y$.

%\smallskip
It remains to prove \eqref{tomte}.
As before we have that $\psi_\e = \psi_\e + g   - g=u_\e +g$,
where $u_\e$ are psh and decrease to $u$ when $\e\to 0$, and $g$ is smooth.
Applying the Bedford-Taylor theory of convergence of Monge-Ampère products
for bounded psh functions, one obtains \eqref{tomte} in case $\mu$ has the form \eqref{na}.
If $\mu=\tau_*\gamma$, where $\tau\colon W\to U$ is proper, then
$\tau^*\psi_\e+\tau^* g$ is psh and smooth on $\tau^{-1}V$ and decreases to $\tau^*\psi+\tau^*g$ when
$\e\to 0$. We conclude that
$(dd^c \tau^*\psi_\e)^k\w \gamma \to (dd^c \tau^*\psi)^k\w \gamma$ and taking $\tau_*$ we get
 \eqref{tomte}.
\end{proof}

% \begin{prop}\label{amazon}
% If $(E,h)$ is \qas\ and $h$ is non-degenerate, then locally there exist smooth metrics $h_\epsilon$,
%decreasing to $h$ when $\epsilon \to 0$, such that
%\begin{equation}\label{hoppla}
%s(E,h_\e) \w \mu \to \s(E,h) \w \mu, \quad \e\to 0, \quad \mu \in \QZ(X).
%\end{equation}
%\end{prop}

%\begin{proof}
%We now  consider (ii).
%\end{proof}

We will now see that, as long as \qas\ bundles have empty degeneracy loci, their Segre forms
in many ways behave as in the smooth case.

\begin{thm}\label{mainLocBdd}
Assume that $E\to X$ is a holomorphic vector bundle equipped with a non-degenerate \qas\ metric
$h$.

\noindent(i)
The Segre endomorphism on $\QZ(X)$ induced by $h$ is independent of the choice of
reference metric, and is denoted by $\s(E,h)$. On the (possibly empty) open subset
where $h$ is smooth, $\s(E,h)$ coincides with multiplication by the Segre form of $h$ there.

\noindent(ii)
Locally there exist smooth metrics $h_\epsilon$,
decreasing to $h$ when $\epsilon \to 0$, such that
\begin{equation}\label{hoppla}
s(E,h_\e) \w \mu \to \s(E,h) \w \mu, \quad \e\to 0, \quad \mu \in \QZ(X).
\end{equation}

\noindent (iii)
If $(E',h')$ is another \qas\ bundle and $h'$ is non-degenerate as well,  then
\begin{equation}\label{tonfisk}
\s(E'\oplus E, h'\oplus h)\w\mu=\s(E,h)\w \s(E',h')\w\mu, \quad \mu\in\QZ(X).
\end{equation}

\noindent (iv)  For any $\ell$ and $k$ the endomorphisms $\s_k(E,h)$ and $\s_\ell(E',h')$ commute. % i.e., $\s_\ell(E',h') \s_k(E,h)=\s_\ell(E',h') \s_k(E,h)$.
\end{thm}

\begin{proof}
The hypothesis means that $Z=\emptyset$ and thus $s(E,h,h_0)=s'(E,h)$. Therefore it is independent of $h_0$,
cf.~\eqref{storlom}. %
It follows by \eqref{segdef}, \eqref{toka} and \eqref{storlom} that if $h$ is smooth on an open set,
then $\s(E,h)$ coincides with multiplication by the Segre form of $h$ there.

We now prove (ii).
Fix a trivialization of $E$, and a local coordinate system $x$ on $X$ such that locally,
$E \cong \C^r_\alpha \times \C^n_x$. Let $\rho_\e$ be a family of radial mollifiers,
and let
\begin{equation} \label{hEpsLocBddDef}
   h_\e = h * \rho_\e + A \e I,
\end{equation}
where $I$ is the identity matrix and $A$ is a large constant.
Note that $h_\e$ depends on both the trivialization of $E$ and the local coordinate system on $X$.
It is readily verified that a convolution of positive definite matrices is positive definite,
and the entries are smooth, so that each $h_\e$ indeed defines a smooth metric on $E$.

Fix a point $x_0\in X$ and a small \nbh $V\subset X$.
Consider the subset of $p^{-1} V\subset\Pk(E)$ where $\alpha_j\neq 0$.
With respect to the frame $\alpha/\alpha_j$ of $L_E$, the
metric $e^{\psi_\e}$ induced by $h_\e$ is represented by the function
$\psi_\e = \log |\alpha/\alpha_j|_{h_\e(x)}$.
Then
$\psi_\e$ is of the form \eqref{eq:locBddQasApproximation} with
$\psi = \log |\alpha/\alpha_j|^2_{h(x)}$,
and $v = \log(|\alpha/\alpha_j|^2)$.
By Proposition~\ref{tunika}, if $A$ is chosen large enough,
$\psi_\e$ is decreasing when $\e\to 0$.
By compactness the same $A$ works where $\alpha_j\neq 0$ for any $j$.  We conclude that
$h_\e$ decreases to $h$.
It follows from Proposition~\ref{tunika} that
\begin{equation}\label{hoppla2}
    s(L_E,e^{\psi_\e}) \w p^* \mu = \sum_k (dd^c \psi_\e)^{k+r-1} \w p^* \mu
    \to \sum_k (dd^c \psi)^{k+r-1} \w p^* \mu.
\end{equation}
Applying $p_*$ we get \eqref{hoppla}, keeping in mind, cf.~\eqref{toka},  that
$\s(E,h_\e) \w \mu=s(E,h_\e)\w\mu$ since $h_\e$ are smooth.  Thus (ii) is proved.

We now consider (iii).
Let $\psi'=\log |\alpha'|^2_{h'(x)}$.
We can assume that $g$ is such that also
$\psi'+g$ is psh. Let  $h'_{\e'}$ be a similar regularization.
It then follows that $\log(e^{\psi'_{\e'}}+e^{\psi_\e})+g$ is psh  for each $\e'$ and $\e$ and decreases to
$\log(|\alpha'|^2_{h'(x)}+|\alpha|^2_{h(x)}) +g$ when $(\e',\e)$ decreases to $(0,0)$.
With the same argument as in the proof of (ii) we conclude that
\begin{equation}\label{kantarell}
s(E'\oplus E, h'_{\e'}\oplus h_\e)\w \mu\to \s(E'\oplus E, h'\oplus h)\w\mu
\end{equation}
unconditionally when $\e',\e\to 0$.
With the product formula for smooth metrics
$$
s(E'\oplus E, h'_{\e'}\oplus h_\e)\w\mu=s(E,h_\e)\w s(E',h'_{\e'})\w\mu.
$$
If we first let $\e'\to 0$ for fixed $\e$ and then let $\e\to 0$ we get \eqref{tonfisk}
from \eqref{kantarell} and \eqref{hoppla}.

Part (iv) requires an additional argument; it is indeed an immediate consequence of Theorem~\ref{main11}~(vii) below.
\end{proof}

Let $E\oplus \cdots \oplus E$ be the direct sum of $k$ copies of $E$.
By repeated use of (iii) we have  $\s(E\oplus\cdots\oplus E, h\oplus\cdots \oplus h)=\s(E,h)^k$ and clearly
 $\s(E\oplus\cdots\oplus E, h_\e\oplus\cdots \oplus h_\e)=\s(E,h_\e)^k$. Applying (ii) to the direct sum, noting that
 $(h\oplus\cdots \oplus h)_\e=h_\e\oplus\cdots\oplus h_\e$,
 we conclude that
 \begin{equation}\label{sotare}
 \s(E,h_\e)^k\to \s(E,h)^k.
 \end{equation}

\smallskip
Since $s(E,h)=\s'(E,h) 1$ it follows, cf.~\eqref{storlom} that
%It follows from the proof that %In particular,  precisely as for a smooth metric we have that
\begin{equation}\label{spagat}
s(E,h)=p_*\sum_{\ell=0}^\infty (dd^c\psi)^\ell,
\end{equation}
which is precisely the  definition \eqref{segdef} when
$h$ is smooth %cf.~\ref{segdef},
but with the
Monge-Ampère products interpreted
in the Bedford-Taylor sense.
Letting $\mu=1$, (ii) implies that %Furthermore,
\begin{equation}\label{slagga}
s(E,h_\e)\to s(E,h).
\end{equation}
Part (iv) means that
$
s(E\oplus E', h\oplus h')=\s(E',h')\w s(E,h).
$

\begin{remark}
If $h$ is  Griffiths negative, then one can take $A=0$ in \eqref{hEpsLocBddDef}.
The resulting approximation, in particular \eqref{slagga},
appeared in \cite{LRRS}.
\end{remark}

%In case $h$ is Griffiths negative, then such a regularization may be obtained
%from \cite[Proposition~3.1]{BePa} by representing $h$ as a matrix in a given
%trivialization of $E$, and letting $h_\e = h \ast \rho_\e$, where $\ast$ denotes
%the convolution with respect to a choice of local coordinates $x$, and $\rho_\e(x)$
%is a family of radial mollifiers, and it will follow by the Bedford-Taylor theory and
%the definition of $\s(E,h)$ that the corresponding currents converge.
%If $h$ is just \qas, then one may instead take
%\begin{equation} \label{hEpsLocBddDef}
%    h_\e = h\circ \rho_\e + A \e I,
%\end{equation}
%where $A$ is chosen to be a large enough constant to make $h_\e$ decreasing in $\e$
%when $\e \to 0$, and $I$ denotes the identity matrix.

%\begin{prop} \label{propEmptyUnboundedLocus}
%Assume that $E\to X$ is a holomorphic vector bundle equipped with a \qas\ metric $h$,
%and assume that the unbounded locus $Z$ of $h$ is empty. %Then the Segre form $s(E,h,h_0)$
%associated to $h$ is independent of the choice of smooth reference metric $h_0$,
%and may hence be denoted $s(E,h)$. Furthermore,
%Then there exists locally a sequence of smooth
%metrics $h_\delta$ decreasing to $h$ such that
%\begin{equation}
 %   \lim_{\delta \to 0} s(E,h_\delta) = s(E,h).
%\end{equation}
%\end{prop}

%\begin{proof}[Proof of Proposition~\ref{propEmptyUnboundedLocus}]
%Proposition~\ref{propEmptyUnboundedLocus} follows immediately from
%Proposition~\ref{mainLocBdd} by taking $\mu=1$.
%\end{proof}

\subsection{The case with a general qas metric $h$}

%We now turn our attention to a general qas  metric.
Here is our main theorem on the Segre endomorphism $\s(E,h,h_0)$.

\begin{thm}\label{main11}
Assume that $E\to X$ is a holomorphic vector bundle equipped with a \qas\ metric $h$ and a smooth reference metric $h_0$, let $Z$ be the degeneracy locus of $h$,
and let $\s(E,h,h_0)$ be the associated Segre endomorphism on $\QZ(X)$.

\smallskip

\smallskip
\noindent
(i) If $h_\e$ is the metric defined by \eqref{regular1},
then
\begin{equation}\label{polo1}
\s(E,h_\e)\w \mu \to \s(E,h,h_0) \w\mu, \quad \e\to 0, \quad \mu\in \QZ(X).
\end{equation}

\smallskip
\noindent
(ii) The endomorphisms $\s_k(E,h,h_0)$ vanish for $k<0$ and $\s_0(E,h,h_0)=I$, the identity mapping.

\smallskip
\noindent
(iii) If $\tau\colon X'\to X$ is a proper mapping,
then $\tau^*h$ is \qas\
and
\begin{equation}\label{sformel1}
\tau_*( \s(\tau^*E,\tau^*h,\tau^* h_0)\w\mu)=\s(E,h,h_0)\w \tau_*\mu, \quad \mu\in\QZ(X').
%\pi_* s(\pi^*E)=s(E).
\end{equation}

\smallskip
\noindent
(iv) For any $\mu\in \QZ(X)$ there is a current $w$ on $X$ such that
\begin{equation}\label{main2222}
(\s(E,h,h_0)- \s(E,h_0))\w\mu = dd^c w.
\end{equation}

\smallskip
\noindent (v)
If $E'$ is another holomorphic vector bundle, equipped with a smooth metric $h'$, then
\begin{equation}\label{main32}
\s(E'\oplus E,h'\oplus h,h'\oplus h_0)\w\mu= \s(E,h,h_0)\w \s(E',h')\w\mu, \quad \mu\in\QZ(X).
\end{equation}

\smallskip
\noindent
(vi)
For each $k$, $\s_k(E,h,h_0)$ yields a well-defined endomorphism on $\QB(X)$,
independent of $h_0$, and
if $h'$ is a \qas\ metric on $E$ comparable to $h$, then
$\s_k(E,h',h_0)$ yields the same endomorphism as $\s_k(E,h,h_0)$ on $\QB(X)$.

\smallskip
\noindent
(vii)
Assume that $(E_1,h_1)$ and $(E_2,h_2)$ are \qas\ bundles, with degeneracy loci $Z_1$ and $Z_2$, respectively, and
$h_{10}$ and $h_{20}$ be smooth reference metrics on these bundles.  Let $X'\subset\subset X$ be an open relatively compact subset.
For any $k,\ell$  and $\mu\in \GZ(Y')$, there is a current $w$ on $X'$, which is the pushforward of a current
on $Z_1 \cap Z_2\cap X'$, such that
$$
    (\s_k(E_1, h_1,h_{10})\w\s_\ell(E_2, h_2,h_{20})- \s_\ell(E_2,h_2,h_{20})\wedge\s_k(E_1,h_1,h_{10})) \wedge \mu = dd^c w.
$$
\end{thm}

%In particular, (vii) means that

%Prata med Richard om detta med stodet for $w$!!!' Moraliskt ar normal med stod pa $V$ samma som $i_*w'$,
%$i\colon V\to X$.

\begin{remark} \label{remMain11}
\noindent (i) If $h_0$ and $h_0'$ are two smooth reference metrics,
%nd $\mu$
%in $\QZ(X)$ with bidegree $(q,q)$,
then by the
first part of (vi),
\begin{equation}\label{trotyl}
\big(\s_k(E,h,h_0) - \s_k(E,h,h_0')) \w \mu, \quad \mu\in \QZ(X),
\end{equation}
belongs to $\QZ_0(X)$.
%$(\s(E,h,h_0)\w \mu - \s(E,h,h_0')) \w \mu$ belongs to $\QZ_0(X)$.
Since $\s(E,h,h_0)$ and $\s(E,h,h_0')$ coincide outside of $Z$,
it follows from Proposition~\ref{elefant} that if $\mu$ has bidegree $(q,q)$, then
%there is a normal current $w$ with support on $Z$
%on $Z$,
%such that
%which is
%the pushforward of a current on $Z$ such that
%In particular, if $\mu$ has bidegree $(q,q)$, then $w$ has bidegree $(k+q-1,k+q-1)$,
%and hence it
\eqref{trotyl} vanishes if $k+q-1 < \codim Z$.

\noindent (ii)
If $\mu$ has bidegree $(q,q)$, then the current $w$ in (vii) has bidegree
$(k+\ell+q-1,k+\ell+q-1)$. Since $w$ is the pushforward of a current on $Z$, it vanishes
if $k+\ell+q -1< \codim Z$.
\end{remark}

\begin{proof}
Notice that if $\chi_\e(t)=\log(e^t+\e)$, then
$$
\psi_\e:=\log(|\alpha|_h^2+\e|\alpha|^2_{h_0})=\log(|\alpha|_h^2/|\alpha|^2_{h_0}+\e)+\log|\alpha|^2_{h_0}=\log(e^{\psi-\psi_0}+\e)+\psi_0 = \chi_\epsilon\circ (\psi-\psi_0) + \psi_0,
$$
so it follows from Theorem~\ref{kusin} that
\begin{equation*}
(dd^c \psi_\e)^k\w p^*\mu\to [dd^c\psi]^k_\omega\w p^*\mu.
\end{equation*}
Applying $p_*$ we get \eqref{polo1} and hence (i) is proved.

Now (ii) follows since it holds when the metric is smooth, and hence also when it
is  \qas\   and non-degenerate by Theorem~\ref{mainLocBdd}~(ii).

Notice that a holomorphic map $\tau \colon X'\to X$ induces a holomorphic map
$\tau\colon \Pk(\tau^*E)\to \Pk(E)$.
Part (iii) now follows from Definition~\ref{kontur} and
Proposition~\ref{ekorre2}~(iii).

\smallskip
In (iv) we can assume that $\mu=\tau_*\gamma$ where $\gamma$ is a product as in \eqref{na}
on $W$ and $\tau\colon W\to X$. In view of (iii) and by changing notation we have to prove
(iv) on $X$ but with $\mu= \gamma$.

If the metric on $E$ is identically $0$, i.e., $Z'=\Pk(E)$, then
$\s(E,h,h_0)=\s(E,h_0)$ and then the statement is trivial.
Let us thus assume that $Z'$ is not equal to $\Pk(E)$.
It may very well happen anyhow that $Z=X$.
Notice that $p^*\gamma$ is a product as in \eqref{na} on $\Pk(E)$.
Let $\psi$ and $\psi_0$ be as above.
Then $\psi-\psi_0$ is a global function on $\Pk(E)$ and, cf.~Proposition~\ref{kamaxel},
\begin{equation}\label{svin}
w'=\frac{\psi-\psi_0}{(1-dd^c\psi_0)(1-\la dd^c\psi\ra)}\w p^*\gamma
\end{equation}
is a locally finite measure. In view of \eqref{bajonett} and \eqref{bajontva} we have
\begin{equation}\label{stare11}
\1_{\Pk(E)\setminus Z'} dd^c w'= \left(\frac{1}{1-\la dd^c\psi \ra}-\frac{1}{1-dd^c\psi_0}\right) \w p^*\gamma
\end{equation}
and
\begin{equation}\label{stare12}
\1_{Z'} dd^c w'=\frac{1}{1-dd^c\psi_0}\1_{Z'} dd^c\psi \frac{1}{1-dd^c\psi} \w p^*\gamma=
s(L_E,e^{\psi_0})\w \sum_{\ell=1}^\infty M^\psi_\ell \w p^*\gamma.
\end{equation}
Notice that $Z'$ has positive codimension so that $M^\psi_0\w p^*\gamma=\1_{Z'} p^*\gamma=0$.
Letting $w=p_*w'$  we conclude from \eqref{segre}, \eqref{stare11} and \eqref{stare12} that
$
dd^c w=(\s(E,h,h_0)-\s(E,h_0))\wedge \mu.
$
Thus (iv) is proved.

\smallskip
We now consider (v).
The metrics $h'$ and $h_\e=h+\e h_0$ are both qas and non-degenerate.
By Theorem~\ref{mainLocBdd}~(iii) therefore
\begin{equation}\label{trall}
    \s(E, h+\e h_0)\w\s(E', (1+\e)h')\w \mu=\s(E'\oplus E, h'\oplus h+\e(h'\oplus h_0))\w\mu.
\end{equation}
Since $s(E', (1+\e)h')=s(E',h')$, (i) now implies that the left hand side of \eqref{trall} tends to $\s(E,h)\w \s(E',h')\w\mu$ and
the right hand side tends to $\s(E'\oplus E, h'\oplus h)\w\mu$.

\smallskip
We now consider (vi).
Let $h_0'$ denote another smooth reference metric.
By \eqref{segre} we see that $\s(E,h,h_0')-\s(E,h,h_0)$ is the image under $p_*$ of
$(s(L_E,e^{\psi_0'})-s(L_E,e^{\psi_0}))\w M^\psi$, which belongs to
$\QZ_0(\Pk(E))$. Now the first statement follows.
The second one follows from
Proposition~\ref{ekorre2}~(ii) and \eqref{torrfoder}.
Thus (vi) is proved. The proof of (vii) is postponed to Section~\ref{kompis}.
\end{proof}

\begin{proof}[Proof of Theorem~\ref{main1}]
Items (i) to (v) are essentially covered by Theorem~\ref{main11} with $\mu=1$.
In the decomposition \eqref{uppesittare} in (iii), one must have that $M^{E,h,h_0} = \1_Z s(E,h,h_0)$.
By \eqref{eq:restrGGZ}, $M^{h,h,h_0}$ belongs to $\QZ(X)$, and in particular, it is closed.
By Theorem~\ref{mainLocBdd} (i), $s'(E,h) = \1_{X\setminus Z} s(E,h,h_0)$ is independent of $h_0$,
and it coincides with the Segre form of $h$ on $X \setminus Z$ if $h$ is smooth there.

The decomposition \eqref{siu0} holds since $M^{E,h,h_0}$ is in $\QZ(X)$.
The statement about the codimension of the non-vanishing multiplicities of
$s_k'$ and $N_k$ follows from Theorem~\ref{superman}.
From Theorem~\ref{main11}~(vi) we know that $s(E,h,h_0)$ and $s(E,h',h_0')$ in (vii) are in the same
class in $\QB(X)$. Hence $M^{E,h,h_0}$ and $M^{E,h',h_0'}$ are as well by \eqref{stare}.
By Proposition~\ref{gorilla} they therefore have the same multiplicities.

We now prove that these numbers are non-negative. %First assume that $\psi$ is psh.
In a small open \nbh $\U$ of a given point on $X$, $E=\U\times \C^r_\alpha$, and therefore
$p^{-1}\U=\U\times\Pk(\C^r_\alpha)$. If we choose a trivial metric $h_0$ on $E$, then
$\omega_\alpha = s_1(L_E,h_0)$ is positive.  %Assume now that $\psi$ is psh.
Assume that the metric is inherited from a morphism $g\colon E\to F$ and $F$
has a trivial metric. Then $\psi=\log|g(x)\alpha|^2$ is psh.
In that case, $s(L_E,e^\psi,e^{\psi_0})$ is a positive current, and hence, so is
$s(E,h,h_0) = p_*(s(L_E,e^\psi,e^{\psi_0}))$. Furthermore, each of the three terms in the decomposition
$s(E,h,h_0) = s'(E,h) + S^{E,h} + N^{E,h,h_0}$ is positive as well, and in particular,
all their components will have non-negative Lelong numbers.
By Lemma~\ref{knepig1} any \qas\ metric $h$ on $E$ is locally equivalent to a psh metric as above.
Thus (vi) is proved. %and therefore all the concerned currents for $E$ with the metric $h$ have non-negative multiplicities.
\end{proof}

\begin{remark}\label{knepig2}
The proof of the positivity of the multiplicities refers to Lemma~\ref{knepig1} which in turn requires that $E$ is  \qas\ and not only that $L_E$ is \qas,
cf. Remark~\ref{viktig}. The remainder of the theorem holds assuming only that $L_E$ is \qas.
\end{remark}

\begin{remark} \label{mainQqasRcas}
Almost all of Theorem~\ref{main1} and Theorem~\ref{main11} have natural extensions
to $\Q$-\qas\ or $\R$-\qas\ metrics. Indeed, if $h$ is $\Q$-\qas\ or $\R$-\qas,
it follows readily that $\s(E,h,h_0)$ yields an endomorphism on $\QZ_\Q(X)$ or $\QZ_\R(X)$,
cf. Remark~\ref{Qqas}. In that case, the multiplicities referred to in the Theorem~\ref{main1} would
be non-negative rational or real numbers, cf.~Remark~\ref{multQ}.
\end{remark}

\subsection{Composition of Segre endomorphisms}\label{kompis}

\begin{ex}\label{grit}
With the notation in the proof of Lemma~\ref{tomte1},
and $\mu=\tau_*\gamma$, we have
\begin{equation}\label{skrot}
\s_k(E,h,h_0)\w \mu=\Pi_*(([dd^c(\tau')^*\psi]_{(\tau')^*\omega}^{k+r-1}\w (p')^*\gamma),
\end{equation}
where the product on the right takes place on $\Pk(\tau^*E)$ and $\Pi=p\circ \tau'=\tau\circ p'$.
To see \eqref{skrot} just notice that
$$
p_*( [dd^c\psi]_\omega^{k+r-1} \w \tau'_* (p')^*\gamma)=p_*\tau'_*( [dd^c(\tau')^*\psi]_{(\tau')^*\omega}^{k+r-1} \w (p')^*\gamma),
$$
and this holds because $\Gamma= (p')^*\gamma$ is a product as in \eqref{na},  and from \eqref{harry2} therefore
$$
[dd^c\psi]^{k+r-1}_\omega \w \tau'_*\Gamma=\tau'_*\big([dd^c(\tau')^*\psi]^{k+r-1}_{(\tau')^*\omega} \w \Gamma\big).
$$
Notice that $\Pk(\tau^* E)=\Pk(E)\times_X W$.
\end{ex}

Let $(E_1,h_1)$ and $(E_2,h_2)$ be two qas bundles over $X$.
Consider \eqref{dia23} with $E=E_2$, $W=\Pk(E_1)$ and $\tau=p_1$. We then get the fiber diagram
\begin{equation}\label{dia24}
\begin{array}[c]{cccc}
Y &  \stackrel{\pi_2}{\longrightarrow}  & \Pk(E_2) \\
\downarrow \scriptstyle{\pi_1} & &  \downarrow \scriptstyle{p_2}  \\
\Pk(E_1) & \stackrel{p_1 }{\longrightarrow}  & X
\end{array}
\end{equation}
with
\begin{equation}%\label{boxer}
Y=\Pk(E_2)\times_{X} \Pk(E_1), \quad
\Pi=p_1\circ\pi_1=p_2\circ \pi_2\colon Y\to X.
\end{equation}
%where $E_1$ and $E_2$ are distinct copies of $E \to X$.
%
 Let $h_{10}$ and $h_{20}$ be smooth reference metrics on $E_1$ and $E_2$ of rank $r_1$ and $r_2$, respectively.
In view of Example~\ref{grit} and \eqref{torrfoder} we have
\begin{multline*}%\label{grisknorr}
\s_{\ell_2}(E_2,h_2,h_{20})\w s_{\ell_1}(E_1,h_1,h_{10})=
(p_2)_*( [dd^c\psi_2]^{\ell_2 +r_2-1}_{\omega_2}\w p_1^*s_{\ell_1}(E,h_1,h_{10}))=\\
(p_2)_*( [dd^c\psi_2]^{\ell_2 +r_2-1}_{\omega_2}\w (\pi_2)_* [dd^c\pi_1^* \psi_1]^{\ell_1 +r_1-1}_{\pi_1^*\omega_1})=
 \Pi_*\big( [dd^c\pi_2^* \psi_2]^{\ell_2 +r_2-1}_{\pi_2^*\omega_2}\w  [dd^c\pi_1^*\psi_1]^{\ell_1 +r_1-1}_{\pi_1^*\omega_1}\big).
\end{multline*}
where $e^{\psi_j}$ and $e^{\psi_{j0}}$ denote the metrics on $L_{E_j}$ induced by $h_j$ and $h_{j0}$, and
$\omega_j=s_1(L_{E_j},e^{\psi_{j0}})$.
In the same way
\begin{equation}\label{august}%
\s_{\ell_2}(E_2,h_2,h_{20})\w \s_{\ell_1}(E_1,h_1,h_{10})\w \mu=
 \Pi_*\big( [dd^c\pi_2^* \psi_2]^{\ell_2 +r_2-1}_{\pi_2^*\omega_2}\w  [dd^c\pi_1^*\psi_1]^{\ell_1 +r_1-1}_{\pi_1^*\omega_1}\w\Pi^*\mu\big),
\end{equation}

\begin{ex}\label{grit1}
Let $(E_1, h_1),\ldots, (E_k,h_k)$ be disjoint copies of $(E,h)$.
Then, similarly,
\begin{equation}\label{2kryss}
\s_{\ell_k}(E,h,h_0)\w\cdots \w\s_{\ell_1}(E,h,h_0)\w \mu=\Pi_*([dd^c\pi_k^*\psi_k]_{\pi_k^*\omega_k}^{\ell_k+r-1}\w
\cdots\w [dd^c\pi_1^* \psi_1]_{\pi_1^*\omega_1}^{\ell_1+r-1}\w \Pi^*\mu \big). %\pi_0^*\gamma\big).
\end{equation}
In particular, our definition of product of Segre forms coincides with the product defined in
\cite[Section~6.1]{LRSW}, cf.~also Remark~\ref{remLRSWcomparison}.
\end{ex}

\begin{lma} \label{s1comm}
Assume that $L_1,L_2$ are \qas\ line bundles over a manifold $Y$, and let $Z=Z_{L_1} \cap Z_{L_2}$ denote the intersection
of their degeneracy loci.  %of $L_1$ and $L_2$.
Let $Y'\subset\subset Y$.
For any $\mu \in \QZ(Y')$ there exists  a current $w$ on $Y'$,
that is the push-forward of a current on $Z\cap Y'$, such that
$$
(\s_1(L_1) \w \s_1(L_2) - \s_1(L_2) \w \s_1(L_1)) \w \mu
= dd^c w.
$$
\end{lma}

For simplicity we write $\s(L_j)$ rather than $\s(L_j, e^{\psi_j}, e^{\psi_{j0}})$.

\begin{proof}
By taking $\mu = \tau_* \gamma$, $\tau \colon W \to Y$, where $\gamma$ is of the form \eqref{na}, we may
assume by \eqref{sformel1}, after replacing $Y$ by $W$, that $\mu$ is of the form \eqref{na}.
Let $\psi_j$ denote the metric on $L_j$, for $j=1,2$, and let $\omega_j$ denote the first Segre form of the smooth reference metric on $L_j$.
Either some of $\psi_j \equiv -\infty$,
say it is the case for $\psi_1$. Then $\s_1(L_1)$ acts by multiplication with the smooth form $\omega_1$,
which commutes with $\s_1(L_2)$, and we may take $w_0 = 0$.
Hence we can assume that $\psi_j \not\equiv -\infty$ for both $j$.

%As in the beginning of Section~\ref{MAhom}, we can take a principalization of the singularities
%of $\psi_j$.
In view of  Proposition~\ref{ekorre2}~(iii) we can replace $Y$ by any  modification of it.  As in the beginning of Section~\ref{MAhom}
% due to Proposition~\ref{ekorre2} (iii),
%since if $\mu$ is of the form \eqref{na}, then it is the push-forward of a corresponding form in any modification.
we may therefore assume that
$$
dd^c \psi_j = [D_j] + [W_j] + \nu_j,
$$
where $D_j,W_j$ are divisors for $j=1,2$, $|D_1|=|D_2|$,
and all irreducible components of $W_1,W_2,D_1$ intersect properly. %have normal crossings. %
Notice that $\s_1(L_1)\w [D_2]\w\mu=\omega_1\w [D_2]\w\mu$ whereas
$$
\s_1(L_1)\w ([W_2] + \nu_2)\w\mu=([D_1]+[W_1]+\nu_1)\w ([W_2]+\nu_2)\w\mu,
$$
where the right hand side can be computed formally since the intersections that occur are proper.
We therefore get %
\begin{multline}\label{januari}
(\s_1(L_1) \w \s_1(L_2) - \s_1(L_2) \w \s_1(L_1)) \w \mu
 = ([D_1] \w ([W_2]+\nu_2-\omega_2) - [D_2] \w ([W_1]+\nu_1-\omega_1))\w \mu.
\end{multline}
If we choose a smooth metric on $\Ok(D_j)$ we get a smooth representative $\alpha_j$
of its Bott-Chern class.  Since $[W_j]+\nu_j-\omega_j$ defines the same class there is a global function $u_j$
such that
$$
dd^c u_j=[W_j]+\nu_j-\omega_j-\alpha_j.
$$
Locally $u_j$ is a potential to $\nu_j$ plus a potential to $[W_j]$ plus a smooth function, and hence it is \qas.
%plus smooth function, and hence it is qas.
%Notice that $u_j$ is qpsh and locally bounded outside $|W_j|$ since a potential to $\nu_j$ is. Moreover it has
%analytic singularities at $|W_j|$ so it is \qas.
It now follows, cf.~the proof of Lemma~\ref{elbas2}, that
$$
dd^c (u_1 [D_2]\w\mu)=[dd^c u_1]\w [D_2]\w\mu= ([W_1]+\nu_1-\omega_1-\alpha_1)\w[D_2]\w\mu.
$$
Thus $w = (u_1 [D_2] - u_2 [D_1]) \wedge \mu
$
is a form with support on $Z=|D_1|=|D_2|$ such that, $dd^c w$ is equal to
\eqref{januari} minus
\begin{equation}\label{januari2}
(\alpha_1 [D_2]-\alpha_2 [D_1])\w\mu.
\end{equation}

It remains to find a $dd^c$-potential to  \eqref{januari2} in $Y'$ that is the push-forward of a current on $Z\cap Y'$.
Since  $Y'\subset\subset Y$,  $Z\cap Y'$ has a finite number of irreducible components
$D^\ell$. % be the irreducible components of $Z\cap Y'$. Since $Y'\subset\subset Y$ there is just a finite number of $D^\ell$.
There are integers $a_j^\ell$ such that
$
D_j=\sum_\ell a_j^{\ell} D^\ell, \ j=1,2.
$
If $\omega^\ell$ are smooth representatives of the Bott-Chern class of $[D^\ell]$, then we can choose
$
\alpha_j=\sum_\ell a_j^\ell\omega^\ell
$
on $Y'$.
Thus
$$
(\alpha_1 [D_2]-\alpha_2 [D_1])\w\mu=\sum_{\ell\neq k}a_1^\ell a_2^k(\omega^\ell[D^k]-\omega^k[D^\ell])\w\mu.
$$
If $v_\ell$ are functions such that $dd^c v_\ell=[D^\ell]-\omega^\ell$, then as before, suppressing $\w \mu$ in the notation,
$$
dd^c (v_\ell [D^k]-v_k[D^\ell])=\omega^\ell[D^k]-\omega^k[D^\ell]-([D^\ell]\w[D^k]-[D^k]\w[D^\ell])=\omega^\ell[D^k]-\omega^k[D^\ell]
$$
since the $ [D^\ell]\w[D^k]$ is a proper intersection and thus commuting. Thus
$w=\sum a_1^\ell a_2^k(v_\ell [D^k]-v_k[D^\ell])$ is a $dd^c$-potential to
\eqref{januari2} as desired. %nd the lemma is proved.
\end{proof}

\begin{proof}[Proof of Theorem~\ref{main11}~(vii)]
Let $Y=\Pk(E_2)\times_X \Pk(E_1)$ be as above, let $L_j\to Y$ be the pullback to $Y$  of the  tautological line bundle $L_{E_j}\to \Pk(E_j)$
and let $Y'$ be the analogous set but over $X'\subset\subset X$.  Then $Y'\subset\subset Y$. For simplicity we write $Y$ rather than $Y'$ in what follows.
By \eqref{august} it is enough to prove the similar statement for $L_j\to Y$, cf.~\eqref{segreL}, and $\ell,k$ replaced by $k+r_2-1, \ell+r_1-1$.
%
%it suffices to prove the result when $(E_1,E_2,k,\ell)$ is replaced
%by $(L_{E_1},L_{E_2},k+\rank E_1 -1,\ell+\rank E_2 -1)$.
%We may thus simply assume that $\rank E_1 = \rank E_2 = 1$.
Since $L_j$ are line bundles,
$\s_k(L_j) = \s_1^k(L_j)$, and so
$$
\big(\s_k(L_1) \w \s_\ell(L_2) - \s_\ell(L_2) \w \s_k(L_1)\big) \w \mu
= (\s_1(L_1)^k \w \s_1(L_2)^\ell - \s_1(L_2)^\ell \w \s_1(L_1)^k) \w \mu.
$$
We may rewrite the right-hand side as a sum of terms of the form
$$
\mathbf{S}^+ \w \big( \s_1(L_1) \w \s_1(L_2) - \s_1(L_2) \w \s_1(L_1)\big) \w \mathbf{S}^- \w \mu,
$$
where $\mathbf{S}^+$ and $\mathbf{S}^-$ are products of $\s_1(L_1)$ and $\s_1(L_2)$.
Let $Z = Z_{L_1} \cap Z_{L_2}$. It suffices to prove that each such
term is of the form $dd^c w$, where $w$ is a current that is the
push-forward of a current on $Z$.
Since $\mathbf{S}^- \w \mu$ is in $\QZ$, we obtain by Lemma~\ref{s1comm} a current
$w_0$, which is the push-forward of a current on $Z$,  such that
$$
\big(\s_1(L_1) \w \s_1(L_2) - \s_1(L_2) \w \s_1(L_1)\big) \w \mathbf{S}^- \w \mu
= dd^c w_0,
$$
where $w_0$ is current that is the push-forward of a current on $Z$.
Note that $w_0$ has support on $Z$, and both $\s_1(L_1)$ and $\s_1(L_2)$ act by multiplication
with a closed smooth form on currents in $\QZ$ with support on $Z$,
and hence $\mathbf{S}^+$ acts as multiplication with a closed smooth form $S^+_0$,
so we may take $w = S^+_0 \w w_0$.
\end{proof}

\section{Chern forms of \qas\ bundles}\label{csection}

Let $(E,h)$ be a \qas\ bundle, $h_0$ a smooth reference metric, and let
$\hat \s(E,h,h_0)=\s(E,h,h_0)-I$.
In this section, we will sometimes drop $h$ and $h_0$ in the notation.
In view of Theorem~\ref{main11}~(ii), $\hat \s(E):=\s(E)-I=\s_1(E)+\s_2(E)+\cdots$
is an endomorphism on $\QZ(X)$ of positive bidegree.
Let $\hat \s(E)^k$ the composition of it $k$ times.
We define the Chern endomorphism
\begin{equation}\label{cherndef}
\c(E,h,h_0):=\sum_{k=0}^\infty (-1)^k \hat \s(E,h,h_0)^k =I -\hat \s(E,h,h_0)+ \hat \s(E,h,h_0)^2+\cdots.
\end{equation}

\begin{prop}\label{main2prop}
Let $E\to X$ be an \qas\ bundle. Then
\begin{equation}\label{trollmor}
%\smallskip
%\noindent
\c(E,h,h_0)\w\s(E,h,h_0)=\s(E,h,h_0)\w\c(E,h,h_0)=I.
\end{equation}
The statements in Theorem~\ref{mainLocBdd} and (ii)-(vi) of Theorem~\ref{main11}
hold if $\s(E,h,h_0)$ is replaced by $\c(E,h,h_0)$.
\end{prop}

\begin{proof}
Clearly $\s(E,h,h_0)$ and $\hat \s(E,h,h_0)$ commute, and hence the computation
$$
\s(E,h,h_0)\w\c(E,h,h_0)=(I+\hat\s(E,h,h_0))\sum_{k=0}^\infty(-1)^k \hat \s(E,h,h_0)^k=I,
$$
is legitimate. The other equality in \eqref{trollmor} follows in the same way.

\smallskip
We now consider the analogue for $\c$ of Theorem~\ref{mainLocBdd}, so assume that $h$ is \qas\ and non-denegerate.
Then clearly $c(E,h)$ is independent of $h_0$.
If  $h_\e$ is a local regularization, then
it follows from \eqref{sotare} that the
analogue for $\c(E,h)$ of %\end{prop}
part (ii) holds.  Part (iii)  implies its analogue for $\c$: Suppressing $h$ and $h'$ in the notation,
by (i) and (iii) for $\s$ we have  $I=\s(E'\oplus E)\w\c(E'\oplus E) =\s(E)\w \s(E')\w \c(E'\oplus E)$.
Multiplying by $\c(E')$ from the right and $\c(E)$ from the left we get
$\c(E)\w \c(E')=\c(E'\oplus E)$.

From \eqref{cformel} we see that
$\c_k(E,h)$ can be expressed as a sum of commuting compositions $\s_{\ell_1}(E,h)\cdots \s_{\ell_m}(E,h)$.
That is,
$$
\c_1(E,h)=-\s_1(E,h), \quad \c_2(E,h)=\s_2(E,h)- \s_1(E,h)^2, \quad etc.
$$
Now (iv) follows for $\c$.

\smallskip

%The corresponding version of Theorem~\ref{mainLocBdd} with $\s(E)$ replaced
%by $\c(E)$ follows from Theorem~\ref{mainLocBdd} and Proposition~\ref{prop:locBddSegreProducts}.

It remains to prove the analogous statements of Theorem~\ref{main11} when
$\s(E)$ is replaced by $\c(E)$.
Part (ii) is obvious, and part (iii) for $\c(E)$ follows from (iii) for $\s(E)$.
To see (iv), notice that $\s(E,h,h_0)-\s(E,h_0)=\hat\s(E,h,h_0)-\hat\s(E,h_0)$ and that $\hat\s(E,h,h_0)$ and $\hat\s(E,h_0)$ commute.
By Theorem~\ref{main11}~(iv), for each $k$, there are $w_k$ such that
$$
dd^c w_k=(\s(E,h,h_0)-\s(E,h_0))\w \sum_{j=0}^{k-1} \hat\s(E,h,h_0)^j \w \hat\s(E,h_0)^{k-1-j}\w\mu.
$$
If $w=-w_1+w_2-\cdots$, then $dd^c w=(\c(E,h,h_0)-\c(E,h_0))\w\mu$.
We get (v) for $\c$ from the statement for $\s$ by the same formal computation as above in the non-degenerate case.
Finally, we get (vi) for $\c$ by repeated use of (vi) for $\s$.
\end{proof}

\begin{proof}[Proof of Theorem~\ref{main2}]
The theorem follows from Proposition~\ref{main2prop} and the fact that the Chern form $c(E,h,h_0) := \c(E,h,h_0) 1$
is in $\QZ(X)$, cf.~Section~\ref{qzsection} in the same way that Theorem~\ref{main1} follows from Theorem~\ref{main11}
and the fact that the Segre form $s(E,h,h_0)$ is in $\QZ(X)$.
 \end{proof}

If  $h$ is \qas\ and non-degenerate, then
by the analogue of Theorem~\ref{mainLocBdd}~(ii), %
$$
c(E,h) = \lim_{\e\to 0} c(E,h_\epsilon)
$$
where $h_\e$ are smooth metrics. %
In particular,
\begin{equation}\label{kloker}
c_k(E,h)=0, \quad k>\rank E,
\end{equation}
since this holds for $c_k(E,h_\e)$.
We do not know whether  $c_k(E,h,h_0)=0$ for $k>\rank E$ if $h$ is a general
\qas\ metric.

\subsection{An alternative endomorphism}
Let $(E,h)$ be a qas bundle on $X$ and a let $h_0$ be a smooth reference metric. We shall now consider an alternative extension
to an endomorphism on $\QZ(X)$ which leads to a different Chern form, that appeared in \cite{Asommar}.
We will use the notation $s(E):=s(E,h,h_0)$ and $s(E_0)=s(E,h_0)$. Let $Z$ be the degeneracy locus of $h$.

We introduce the endomorphism
\begin{equation}\label{szdef}
\s(E,Z)\w\mu=\s(E)\w \1_{X\setminus Z}\mu+\s(E_0)\w \1_Z \mu, \quad \mu\in \QZ(X).
\end{equation}
If $Z=X$, then it is just $\s(E_0)$ so it is of interest only when $Z$ has positive codimension,
which we will assume from now on. Note that if $s(E,Z) := \s(E,Z) 1$, then
\begin{equation} \label{eq:sEZ}
   s(E,Z) = s(E)
\end{equation}
since $\1_Z 1 = 0$ and $\1_{X\setminus Z} 1 = 1$ as long as $Z$ has positive codimension.

\begin{prop}
Let $E\to X$ be an \qas\ bundle and $\mu\in \QZ(X)$.

 \smallskip
\noindent
(i)
\begin{equation} \label{eq:sEZ2}
  (\s(E,Z)-\s(E)) \wedge \mu = (\s(E)-\s(E_0)) \wedge \1_Z \mu.
\end{equation}

\noindent (ii)
There is a current $w$, which is the pushforward of a current on $Z$, such that
\begin{equation} \label{eq:sEZ3}
    (\s(E,Z)-\s(E)) \wedge \mu = dd^c w.
\end{equation}
\end{prop}

\begin{proof}
Equation \eqref{eq:sEZ2} follows directly from the definition of $\s(E,Z)$.
To prove (ii), first notice that since $\1_Z \mu$ is in $\QZ(X)$ and has support on $Z$,
by Proposition~\ref{tax}, there is a current $\mu_Z \in \QZ(Z)$ such that $\1_Z \mu = i_* \mu_Z$, where $i \colon Z \to X$ denotes the inclusion.
By Theorem~\ref{main11} (iv) there is $w' \in \QZ(Z)$ such that
$$
    (\s(i^*E)-\s(i^* E_0)) \wedge \mu_Z = dd^c w'.
$$
To be precise, we have formulated Theorem~\ref{main11} only for a complex manifold, but
the proof if parts (iii) and (iv) works just as well if $Z$ is any reduced analytic space. %singular.
%when want to apply it here on $Z$ which might be singular, i.e., a reduced analytic space.
%However, one may note either that the same result holds on reduced analytic spaces,
%with the same proof, or one might reduce to the case of manifolds by taking a resolution
%of singularities $\tau \colon Z' \to Z$, finding the appropriate current on $Z'$
%and taking the push-forward to $Z$.
Hence, \eqref{eq:sEZ3} follows by (i) together with Theorem~\ref{main11} (iii) with $w = i_* w'$.
\end{proof}

It follows from \eqref{szdef} that $\s_0(E,Z)\w\mu=\mu$ so that $\s_0(E,Z)=I$, the identity on
$\QZ(X)$.
We can thus define
\begin{equation}\label{hatdef}
\c(E,Z)\w\mu=\sum_{k=0}^\infty (-1)^k \hat \s(E,Z)^k\w\mu,
\end{equation}
where $\hat \s(E,Z)=\s(E,Z)-I$.

\begin{prop}\label{main22}
Let $E\to X$ be an \qas\ bundle and $\mu\in \QZ(X)$.

 \smallskip
\noindent
(i)
 \begin{equation}\label{baka}
\c(E,Z)\w\s(E,Z)=\s(E,Z)\w\c(E,Z)=\I.
\end{equation}

\smallskip
\noindent
(ii) For each $\mu\in\QZ(X)$ there is a current $w$, which is the pushforward of a current
on $Z$, such that
$$
dd^c w=\big(\c(E,Z)-\c(E)\big)\w\mu.
$$

\smallskip
\noindent
(iii) If $Z$ has positive codimension in $X$, then, if $\c'(E):= \1_{X\setminus Z}\c(E)$,
\begin{equation}\label{sommarform}
\c(E,Z)\w\mu=\big(\c'(E) - \c(E_0) \w M^E\w \c'(E)\big)\w \1_{X\setminus Z}\mu+\c(E_0)\w\1_Z\mu.
\end{equation}
\end{prop}

By (ii),  $c(E,Z):=\c(E,Z)\w 1$ is cohomologous with $c(E)$ and hence it defines the expected
Bott-Chern class.

By definition both $c(E)$ and $c(E,Z)$ are equal to the usual Chern form in $X\setminus Z$
if the metric on $E$ is smooth there.
However, in general $c(E,Z)\neq c(E)$ across $Z$, see Example~\ref{chernex}.
%%%

%

\begin{remark}
In view of (iii),  $c(E,Z)$ is precisely the Chern form that appeared in \cite{Asommar} in case
the metric on $E$ is obtained  from $F$ via a generically injective $g\colon E\to F$,
  cf.~Example~\ref{favvis0}
\end{remark}

\begin{proof}[Proof of Proposition~\ref{main22}]
Part (i) follows in the same way as (i) in Proposition~\ref{main2prop}.
To see (ii) we first write
$$
    \hat\s(E)^k-\hat\s(E,Z)^k = \sum_{\ell=0}^{k-1} \s(E,Z)^{k-\ell-1} \w (\s(E)-\s(E,Z)) \w \s(E)^\ell.
$$
By \eqref{eq:sEZ3}, for each $\ell$ there exists $w_\ell$, which is the push-forward of
a current on $Z$, such that
$$
    (\s(E)-\s(E,Z)) \w \s(E)^\ell \w \mu = dd^c w_\ell.
$$
Since $w_\ell$ has support on $Z$,
$$
    \s(E,Z)^{k-\ell-1} \w (\s(E)-\s(E,Z)) \w \s(E)^\ell \w \mu = s(E_0)^{k-\ell-1} \w dd^c w_\ell
    = dd^c ( s(E_0)^{k-\ell-1} w_\ell),
$$
so we get (ii) with
$$
    w = \sum s(E_0)^{k-\ell-1} \w w_\ell.
$$

\smallskip
For (iii), recall that
$M^E \w \mu:=\1_Z \s(E)\w\mu$ and $\s'(E)\w\mu=\1_{X\setminus Z}\s(E)\w\mu.$
One checks that %Notice that
\begin{equation}\label{tokfan}
\hat \s(E,Z)= \hat \s(E)\w \1_{X\setminus Z}+\hat s(E_0)\w \1_Z
\end{equation}
We will prove by induction that
\begin{equation}\label{snuva}
\hat \s(E,Z)^k=\hat \s'(E)^k+\sum_{\ell=0}^{k-1}\hat \s(E_0)^\ell \w M^E \w \hat \s'(E)^{k-1-\ell}, \quad k=1,2, \ldots,
\end{equation}
when acting on $\1_{X\setminus Z} \mu$.
By definition it holds for $k=1$. Assume that it holds for $k$. When applying \eqref{tokfan}
to the right hand side of \eqref{snuva} the first mapping only acts on the first term and gives
$\hat \s(E)\w \hat \s'(E)^k= \hat \s'(E)^{k+1}+ M^E \w \hat \s'(E)^k$. The second mapping neglects the first term
and just gives another factor $\hat \s(E_0)$ on the terms in the sum. Thus \eqref{snuva} holds.
Now
\begin{multline*}
\c(E,Z)=I+\sum_{k=1}^\infty (-1)^k\Big(\hat \s'(E)^k+\sum_{\ell=0}^{k-1}\hat \s(E_0)^\ell \w M^E \w\hat \s'(E)^{k-1-\ell}\Big)=\\
\c'(E)+\sum_{\ell=0}^\infty \sum_{k=\ell+1}^\infty (-1)^k\hat \s(E_0)^\ell \w M^E \w \hat \s'(E)^{k-1-\ell}=\\
\c'(E)-\sum_{\ell=0}^\infty (-1)^\ell \s(E_0)^\ell \w M^E \w \hat \c'(E)=\c'(E)-\c(E_0) \w M^E \w \c'(E)
\end{multline*}
when acting on $\1_{X\setminus Z} \mu$.
The action on $\1_Z\mu$ is just $\hat \s(E_0)^k \w\1_Z\mu$. Thus (iii) follows.
\end{proof}

%\begin{remark}
%One can check directly that $c(E,Z)=\c(E,Z) 1$ is cohomologous with $c(E_0)$. In fact, accepting that the operations are legitimate
%we have, if
%$$
%dd^c u= s(E)-s(E_0)=M^E+s'(E)-s(E_0),$$
%that
%$$
%-dd^c (c(E_0) \w u \w c'(E))=- c(E_0)\w(s(E)-s(E_0))\w c'(E)= c'(E)-c(E_0)\w M^E\w c'(E)-c(E_0).
%$$
%\end{remark}

\section{Remarks and examples}\label{remex}
If $(E,h)$ is qas and its degeneracy locus $Z$ is empty, then, in view of \eqref{spagat} and \eqref{slagga}, the Segre form $s(E,h)$ defined here seems to be the only reasonable one and thus canonical.

If $Z\neq\emptyset$, then any closed current of order $0$ that coincides with $s_k(E,h)$ on $X\setminus Z$ is unique
if $k<\k:=\codim Z$.
This follows from the dimension principle for normal currents, see, e.g., \cite[Corollary III.2.11]{Dem}.
In particular, $s_k(E,h) = s_k'(E,h)$, which coincides with the currents considered by Xia
(for more general classes of metrics), \cite{Xia}, based on the non-pluripolar Monge-Amp\`ere product.
Since the  class of $s_k(E,h,h_0)$ in $\QB(X)$ is independent of $h_0$ it follows from Proposition~\ref{elefant}
that $s_\k(E,h,h_0)$ is independent of $h_0$. % even for $k=p$.
If the metric $h$ in addition is Griffiths
negative, then by \cite{LRRS}, $s_k(E,h)$ may be obtained as $\lim_{\e \to 0} s_k(E,h_\e)$
for any sequence of smooth Griffiths negative metrics $h_\e$ decreasing to $h$ as $\e \to 0$
for any $k \leq \k$.
On the other hand, \cite[Example~8.2]{LRRS} shows that $s_k(E,h,h_0)$ might depend
on $h_0$ for any $k > \k$.  It follows from Theorem~\ref{main1} (i)
that the limits $s(E,h_\e)$ as $\e \to 0$ might depend on the sequence $h_\e$, even
when $h_\e$ is a sequence of Griffiths negative smooth metrics decreasing to $h$.
Therefore it does not seem to be any  'canonical' choice of  Segre form $s_k(E,h)$ for $k>\k$.

Similar statements as above  hold also for products
$$
s_{k_t}(E,h) \wedge \dots \wedge s_{k_1}(E,h) :=
\s_{k_t}(E,h) \wedge \dots \wedge \s_{k_1}(E,h) 1,
$$
depending on how $k=k_1+\dots+k_t$ compares to $\k$. Such a product is furthermore robust
in the sense that it is commutative in the factors provided $k \leq \k$, which follows
from Theorem~\ref{main11} (vii), cf.~Remark~\ref{remMain11} (ii).

By Theorem~\ref{mainLocBdd}~(iii), $\s(E,h+\e h_0)^k\mu=
s(E\oplus\cdots \oplus E, h\oplus\cdots \oplus h +\e (h_0\oplus \cdots \oplus h_0)\w\mu$
for each $k$. Therefore $\lim_{\e\to 0} c(E, h+\e h_0)\w\mu$ exists and is in $\QZ(X)$. Moreover
it coincides with $c(E,h)\w\mu$ outside the degeneracy locus.  However, in general it is not equal to
$\c(E,h,h_0)\w\mu$. This will be discussed in a forthcoming paper.

\subsection{Difference between $c_k(E)$ and $c_k(E,Z)$}

%\begin{ex}
In general $c(E,Z)$ and $c(E)$ are different. The simplest example is when $Z=X$ and $Z'\neq Z$. Then in general
$s(E)$ will have a residue term, whereas $s(E,Z)=\omega$ and hence $c(E,Z)=c(E_0)$.
%\end{ex}

\begin{ex}\label{chernex}
Let $g$ be the $2\times 2$-diagonal matrix, with entries $x_1$ and $x_2$,
considered as a morphism between the trivial bundles $E=X\times \C^2_\alpha=F$,
and $F$ is equipped with the trivial metric.
With the trivial metric on $E$ as reference metric $h_0$ and the singular metric $h$ induced by $g$,
we have
\begin{equation}
   s_1(E) = [x_1 = 0] + [x_2 = 0].
\end{equation}
This follows by direct computation, or, e.g., from \cite[Proposition~11.1]{Asommar}.
From \cite[Example~11.2]{Asommar} we have that $s_2(E)=[x_1=x_2=0]$.

Let us now compute $c(E)$ and $c(E,Z)$ for the same \qas\ bundle $E$.
% or by direct computation.
To begin with $c_1(E)=c_1(E,Z)=-s_1(E)=-( [x_1 = 0] + [x_2 = 0])$.
A calculation yields that
\begin{equation}
\s_1(E)^2 = \s_1(E)\w ([x_1 = 0] + [x_2 = 0])= 2[x_1=x_2=0].
\end{equation}
For instance, if $\tau \colon \C\to \C^2$, $t\mapsto(t, 0)$, then $\s_1(\tau^*E)=[t=0]$.
Hence, by Theorem~\ref{main11}~(iii),
$\s_1(E)\w[x_2=0]= \tau_*(s_1(\tau^*E))=\tau_*[t=0]=[x_1=x_2=0]$.
From \cite[Example~11.2]{Asommar} we have that $s_2(E)=s_2(E,Z)=[x_1=x_2=0]$. Thus
$$
c_2(E)=s_1(E)^2-s_2(E)=[x_1=x_2=0].
$$
Since $Z=\{ x_1x_2 = 0 \}$, and $s_1(E) = s_1(E,Z)$
has support on $Z$, cf.~above and \eqref{eq:sEZ},
$$
s_1(E,Z)^2=\s_1(E,Z)\w  ([x_1 = 0] + [x_2 = 0])= 0,
$$
since $\s_1(E_0) = 0$.
We conclude that
$$
c_2(E,Z)=-[x_1=x_2=0],
$$
which is different from $c_2(E)$.
\end{ex}

\def\listing#1#2#3{{\sc #1}:\ {\it #2},\ #3.}

\end{document}